\documentclass[reqno,final]{siamltex}
\usepackage{amscd}
\usepackage{amsfonts}
\usepackage{amsmath}
\usepackage{amssymb}
\usepackage{amsxtra}
\usepackage{booktabs}
\usepackage{bbm}
\usepackage{cite}
\usepackage{dcolumn}
\usepackage{diagbox}
\usepackage{epic}
\usepackage{exscale}
\usepackage{float}
\usepackage{geometry}
\usepackage{graphicx}
\usepackage[hidelinks]{hyperref}
\usepackage{ifthen}
\usepackage{ifpdf}
\usepackage{latexsym}
\usepackage{layouts}
\usepackage{mathrsfs}
\usepackage{multicol}
\usepackage{multirow}
\usepackage{pifont}
\usepackage{psfrag}
\usepackage[shortlabels]{enumitem}
\usepackage{slashbox}
\usepackage{stmaryrd}
\usepackage{subfigure}
\usepackage{tabularx}
\usepackage{tikz}
\usepackage{txfonts}
\usepackage{verbatim}
\usepackage{wrapfig}
\usepackage{xcolor}

\providecommand{\Div}{\operatorname{div}}          

\providecommand*{\Dist}[2]{\operatorname{dist}({#1};{#2})}   
\providecommand*{\Dist}[2]{\Dist{#1}{#2}}
\providecommand{\Supp}{\operatorname{supp}}                            
\providecommand{\supp}{\Supp}

\renewcommand{\Re}{\operatorname{Re}}             
\renewcommand{\Im}{\operatorname{Im}}             

\newcommand{\Vi}{{\mathbf{i}}}

\newcommand{\Bd}{{\boldsymbol{d}}}

\newcommand{\Bn}{{\boldsymbol{n}}}

\newcommand{\Bp}{{\boldsymbol{p}}}

\newcommand{\Bx}{{\boldsymbol{x}}}
\newcommand{\By}{{\boldsymbol{y}}}
\newcommand{\Bz}{{\boldsymbol{z}}}

\newcommand{\BL}{{\boldsymbol{L}}}

\newcommand{\xibf}{\boldsymbol{\xi}}

\newcommand{\Cc}{{\cal C}}

\newcommand{\Ce}{{\cal E}}

\newcommand{\Ch}{{\cal H}}

\newcommand{\Cl}{{\cal L}}
\newcommand{\Cm}{{\cal M}}

\newcommand{\Cp}{{\cal P}}

\newcommand{\Ct}{{\cal T}}
\newcommand{\Cu}{{\cal U}}
\newcommand{\Cv}{{\cal V}}

\newcommand{\bbA}{\mathbb{A}}

\newcommand{\bbC}{\mathbb{C}}

\newcommand{\bbR}{\mathbb{R}}

\newcommand{\bbZ}{\mathbb{Z}}

\providecommand*{\wt}[1]{\widetilde{#1}}
\providecommand*{\wh}[1]{\widehat{#1}}

\newcommand*{\N}[1]{\left\|{#1}\right\|}     
\newcommand*{\sTN}[1]{|\hspace{-0.1mm}\|{#1}\|\hspace{-0.1mm}|}     
\newcommand*{\SN}[1]{\left|{#1}\right|}      

\newcommand*{\TNHone}[2][\defaultdomain]{{\sTN{#2}}_{\Hone[{#1}]}}

\newcommand*{\Lp}[2][\defaultdomain]{L^{#2}({#1})}

\newcommand*{\NLp}[3][\defaultdomain]{\N{#2}_{\Lp[#1]{#3}}}

\newcommand*{\Ltwo}[1][\defaultdomain]{\Lp[#1]{2}}

\newcommand*{\NLtwo}[2][\defaultdomain]{\NLp[#1]{#2}{2}}

\newcommand*{\Hm}[2][\defaultdomain]{H^{#2}({#1})}

\newcommand*{\bHm}[3][\defaultdomain]{H_{#3}^{#2}({#1})}

\newcommand*{\Hone}[1][\defaultdomain]{\Hm[#1]{1}}

\newcommand*{\zbHone}[1][\defaultdomain]{\bHm[#1]{1}{0}}

\newcommand*{\NHone}[2][\defaultdomain]{{\N{#2}}_{\Hone[{#1}]}}

\newcommand*{\SNHone}[2][\defaultdomain]{{\SN{#2}}_{\Hone[{#1}]}}

\newcommand*{\Hdiv}[1][\defaultdomain]{\boldsymbol{H}(\Div,{#1})}

\newcommand*{\jump}[2][]{\left \llbracket{#2}\right\rrbracket_{#1}}

\newcommand{\D}{\mathrm{d}}

\newcommand{\ol}{\overline}

\newcommand{\be}{\begin{eqnarray}}
\newcommand{\ee}{\end{eqnarray}}

\newcommand{\ben}{\begin{eqnarray*}}
\newcommand{\een}{\end{eqnarray*}}

\def\address#1{\expandafter\def\expandafter\@aabuffer\expandafter
    {\@aabuffer{\affiliationfont{#1}}\relax\par
    \vspace*{13pt}}}

\def\X0{V(\Cm,\boldsymbol{0})}

\title{Wavenumber-explicit stability and preasymptotic error analysis of UPML finite element method for obstacle scattering problems
\thanks{This work is supported in part by the Strategic Priority Research Program of the Chinese Academy of Sciences (XDB0640000, XDB0640300) and the National Key R \& D Program of China (2025YFA1016600, 2025YFA1016602).}}

\author{
   Yuhao Wang \thanks{School of Mathematical Science, University of Chinese Academy of Sciences;
    Academy of Mathematics and Systems Science, Chinese Academy of Sciences,
    Beijing, 100190, China.
    (wangyuhao@amss.ac.cn)
  }
  \and Weiying Zheng \thanks{State Key Laboratory of Mathematical Sciences, Institute of Computational Mathematics and Scientific/Engineering Computing,
    Academy of Mathematics and Systems Science,
    Chinese Academy of Sciences;
    School of Mathematical Science,
    University of Chinese Academy of Sciences,
    Beijing, 100190, China.
    (zwy@lsec.cc.ac.cn)
  }
}

\begin{document}
\maketitle

\begin{abstract}
This paper develops a wavenumber-explicit stability and preasymptotic error analysis for finite element approximation of the two-dimensional Helmholtz scattering problem with a uniaxial perfectly matched layer (UPML) truncation. The analysis is based on  direct estimates of the stretched Green kernel associated with the Cartesian complex coordinate transformation. We establish explicit stability for the truncated UPML problem. In particular, we prove that the inf-sup constant of the truncated UPML formulation is $\mu_L = O(k^{-1})$ in a natural $k$-weighted $H^1$ norm. As a consequence, we also obtain an exponential decay estimate for the PML truncation error. Based on this stability estimate, we formulate a linear continuous interior penalty finite element method (CIP-FEM) on the truncated UPML domain. A key ingredient of the analysis is to establish a piecewise $H^{1+s}$-regularity result (for any $0<s<1$), which accounts for coefficient jumps across PML interfaces and singularities induced by Cartesian corner geometries. This fractional regularity is sufficient to derive wavenumber-explicit preasymptotic error estimates. Numerical experiments confirm the theoretical predictions and illustrate the effectiveness of the UPML CIP-FEM in the high-frequency regime.
\end{abstract}



\section{Introduction}

The Helmholtz equation plays a fundamental role in time-harmonic wave propagation problems arising in acoustics, electromagnetics, and related applications. In the high-frequency regime, numerical approximation of the Helmholtz equation becomes challenging due to the strong indefiniteness of the operator and the presence of pollution effects in standard finite element discretizations.

In this paper, we study the acoustic scattering problem by a bounded and starlike obstacle $D\subset \bbR^2$ which has a $C^2$-smooth boundary $\Gamma=\partial D$
\begin{subequations}\label{model}
\begin{align}
-\Delta{u} - k^2u =f \quad &\text{in} \;\;D^c :=\bbR^2\backslash\ol{D}, \label{model-eqn}\\
u=0 \quad &\text{on} \;\; \Gamma, \label{model-bc} \\
\sqrt{r}\big(\partial_r u-\Vi k u\big)
\to 0 \quad & \text{as}\;\; r=|\Bx|\to\infty, \label{model-rad}
\end{align}
\end{subequations}
where $f\in\Ltwo[\bbR^2]$ is a distributional source which has compact support, $k$ is the constant wavenumber, and \( \partial_r u\) represents the partial derivative in the radial direction.  Since we are interested in the high wavenumber regime, $k\ge 1$ is assumed throughout this paper. 

Wavenumber-explicit stability and error estimates play a fundamental role in the numerical analysis of high-frequency scattering problems. In the one-dimensional setting, Ihlenburg and Babu\v{s}ka~\cite{IhlenburgBabuska1995,IhlenburgBabuska1997} proved that the discrete inf-sup condition holds with a constant of order \(O(k^{-1})\) and derived corresponding pre-asymptotic error estimates for the finite element method. The extension of such analyses to higher dimensions is substantially more challenging. Most existing results concern Helmholtz problems equipped with impedance boundary conditions; see, for example, \cite{DuWu2015,demkowicz2012wavenumber,feng2009discontinuous,feng2011hp,MelenkSauter2011,melenk2013general,shen2007analysis,Wu2014,ZhuWu2013}. In contrast, analogous results for exterior scattering problems remain relatively limited, primarily because establishing wavenumber-explicit stability estimates with the optimal inf-sup scaling \(O(k^{-1})\) is considerably more difficult. Chandler-Wilde and Monk~\cite{ChandlerWildeMonk2008} investigated the exterior scattering problem for a bounded sound-soft obstacle. By introducing a Dirichlet-to-Neumann (DtN) boundary condition, they reformulated the problem on a truncated computational domain and showed that, when the obstacle is starlike, the inf-sup constant is of order \(O(k^{-1})\).

The perfectly matched layer (PML) method was first introduced by B\'erenger~\cite{Berenger1994} and was subsequently reformulated through complex coordinate stretching~\cite{ChewWeedon1994}. It provides a flexible framework for replacing the radiation condition at infinity with an absorbing layer posed on a bounded computational domain. For scattering problems involving anisotropic or geometrically complicated obstacles, the uniaxial, or Cartesian, PML is particularly attractive, since it permits the artificial truncation boundary to be placed close to the obstacle and facilitates a convenient partition of the computational domain~\cite{BramblePasciak2013,ChenWu2008,ChenZheng2010}.

Wavenumber-explicit analyses of PML methods were only developed much later. Li and Wu~\cite{LiWu2019} established the inf-sup stability of the radial PML method for media with piecewise constant coefficients and showed that the corresponding inf-sup constant is of order \(O(k^{-1})\). They also derived pre-asymptotic error estimates for both the standard finite element method and the continuous interior-penalty finite element method (CIP-FEM). Subsequently, Chaumont-Frelet, Gallistl, Nicaise, and Tomezyk~\cite{ChaumontFrelet2022} investigated the radial PML method for two-dimensional Helmholtz problems and proved an \(O(k^{-1})\) inf-sup estimate for a class of monotonically increasing medium coefficients.

In contrast, the wavenumber-explicit stability theory for uniaxial PML (UPML) methods remains far less developed and has so far been concerned mainly with source-scattering problems with $D=\emptyset$. In 2013, Chen and Xiang obtained an inf-sup stability estimate with a constant of order \(O(k^{-3/2})\) \cite{ChenXiangSTDDM_PartI}. Their analysis  was primarily motivated by the source transfer domain decomposition method. In 2024, Galkowski, Gong, Graham, Lafontaine, and Spence posted a preprint on arXiv, which remains unpublished, investigating overlapping Schwarz methods for the Helmholtz equation \cite{gal2024}. Their work provides a wavenumber-explicit analysis based on the assumption that the UPML is $C^\infty$-smooth. Consequently, the $O(k^{-1})$ inf-sup stability has not yet been established for the commonly used piecewise polynomial UPMLs in source-scattering problems, let alone for those in obstacle-scattering problems.

Another important issue is the pre‑asymptotic error analysis of finite element methods for the UPML‑truncated problem. This analysis hinges critically on regularity estimates for the PML problem. For a circular PML with a piecewise linear coordinate transformation, both the PML interface and the truncation boundary are smooth, and the solution is piecewise $H^2$-regular \cite{LiWu2019}. In the case of piecewise linear UPML, however, the coefficients of the PML equation are discontinuous across Cartesian interfaces, so the solution no longer possesses piecewise \(H^2\)-regularity. To overcome this difficulty, we prove that the PML solution is piecewise \(H^{1+s}\)-regular for any \(0<s<1\). Building upon this regularity result, we then derive pre‑asymptotic error estimates for the finite element approximation.

The present paper has two main objectives. The first is to establish a wavenumber-explicit stability estimate for the UPML formulation of~\eqref{model}, and the second is to derive pre-asymptotic error estimates for its finite element discretization. To this end, we develop a new technique for analyzing the associated Green's kernel and derive wavenumber-explicit estimates for both the source- and obstacle-scattering problems. We prove that the truncated UPML formulation satisfies the inf-sup condition with a constant of the optimal order \(O(k^{-1})\). We also establish exponential convergence of the UPML solution to the exact solution as the layer parameters increase. Building on these stability and convergence results, we discretize the truncated problem by the CIP-FEM and obtain pre-asymptotic error estimates for the numerical solution. Finally, numerical experiments are presented to validate the theoretical findings.

The remainder of the paper is organized as follows.
In section~\ref{sec:upml-prelim}, we introduce the UPML by using a complex coordinate transformation, and propose the UPML problem on the truncated domain.
In section~\ref{sec:green}, we propose a dyadic estimate for the transformed Green's kernel.
In section~\ref{sec:source}, we prove the inf-sup condition for the truncated source problem with an inf-sup constant being $O(k^{-1})$.  In section~\ref{sec:obstacle}, we establish the inf-sup stability of the UPML method for obstacle-scattering problems. In section~\ref{sec:CIP-fem}, we derive  the pre-asymptotic error estimate for the CIP-FEM approximation of the truncated UPML problem. Section~\ref{sec:numerics} reports
numerical results.

\section{The uniaxial PML method}\label{sec:upml-prelim}

The purpose of this section is to introduce the UPML method for truncating the exterior scattering problem \eqref{model}. The construction follows the standard Cartesian complex-stretching framework, as used for example in \cite[Section~2.1]{ChenXiangSTDDM_PartI}.

\subsection{Notations}

For a bounded Lipschitz domain \(\Omega\subset \bbR^2\), we write
\[
(u,v)_{\Omega}:=\int_{\Omega} u v, \quad
\NLtwo[{\Omega}]{v} := (v,\bar{v})_{\Omega}^{1/2}
\]
for the space $\Ltwo[{\Omega}]$ of complex-valued functions, and equip \(H^1({\Omega})\) with the \(k\)-weighted norm
\[
\sTN{v}_{H^1({\Omega})}:= \big(
\|\nabla v\|_{L^2({\Omega})}^2
+k^2\|v\|_{L^2({\Omega})}^2\big)^{1/2}.
\]
The dual space of a Banach space $V$ is denoted by $V'$. It is equipped with the norm
\[
\N{\ell}_{V'} := \sup_{0\ne v\in V}\frac{|\langle \ell,v\rangle|}{\N{v}_V}
\quad \forall\, \ell\in V'.
\]
The subspace of $\Hone[{\Omega}]$ with homogeneous boundary conditions is denoted by $\zbHone[{\Omega}]$, and its dual space is denoted by \(H^{-1}({\Omega})\).
Unless stated otherwise, all \(H^1\)- and \(H^{-1}\)-norms below are understood in this weighted sense. Whenever a dual space of the form \((H^1({\Omega}))'\) appears later, it is always the dual of \(H^1({\Omega})\) endowed with the same \(k\)-weighted norm.
In particular, we define
\[
\sTN{\ell}_{{\Hone}'} := \sup_{0\ne v\in\Hone}\frac{|\langle \ell,v\rangle|}{\TNHone{v}}
\quad \forall\, \ell\in \Hone'.
\]
Moreover, all generic constants denoted by \(C\) are independent of \(k\), the mesh size \(h\), and the source term $f$.

It is known that $G(\Bx,\By)=\frac{\Vi}{4}H_0^{(1)}(k|\Bx-\By|)$ is the fundamental solution of the free-space Helmholtz equation, where $H^{(1)}_0$ denotes the zero-order Hankel function of the first kind. Then the solution of problem \eqref{model} admits the integral representation
\begin{align} \label{u-integ}
u(\Bx)=\int_{D^c}G(\Bx,\By)f(\By){\D}\By
+\int_\Gamma \frac{\partial u}{\partial\Bn}(\By)\, G(\Bx,\By) \D s(\By).
\end{align}

We choose $l_1,l_2>0$ such that $D$ and $\supp(f)$ are contained in the rectangle
\begin{align}\label{eq:Bl}
B_l :=  \big\{\Bx=(x_1,x_2)^T\in\bbR^2:\ |x_1|<l_1,\ |x_2|<l_2\big\}.
\end{align}
Write $\Gamma_l :=\partial B_l$ and $\Omega_l:=B_l\cap D^c$. By \cite{ChandlerWildeMonk2008}, problem ~\eqref{model} has a unique solution. From Corollary~3.1 and Theorem~3.2 of \cite{Li-Zheng-Zhu-2020}, there exists a generic constant $C$ independent of $k$ such that
\begin{align} \label{eq:stab-u}
\sTN{u}_{\Hone[\Omega_l]} + k^{-1}\SN{u}_{H^2(\Omega_l)}
\le C\NLtwo[\Omega_l]{f}.
\end{align}

\subsection{Uniaxial complex stretching}

To introduce the UPML, we fix $d_1,d_2>0$ and set
\begin{align}\label{eq:BL-pml}
L_1:=l_1+d_1, \quad L_2:=l_2+d_2, \quad
B_L:=(-L_1,L_1)\times(-L_2,L_2),\quad
\Gamma_L:=\partial B_L.
\end{align}
The domain $B_L\setminus \overline{B_l}$ is regarded as the wave-absorbing layer (Fig.~\ref{fig:upml-geometry}). For convenience, we shall use the notations
\[
B_l^c:=\bbR^2\backslash\ol B_l,\quad
B_L^c:=\bbR^2\backslash\ol{B}_L,\quad
\Omega_L = D^c\cap B_L.
\]

\begin{figure}[htp!]
\centering
\includegraphics[width=0.4\textwidth]{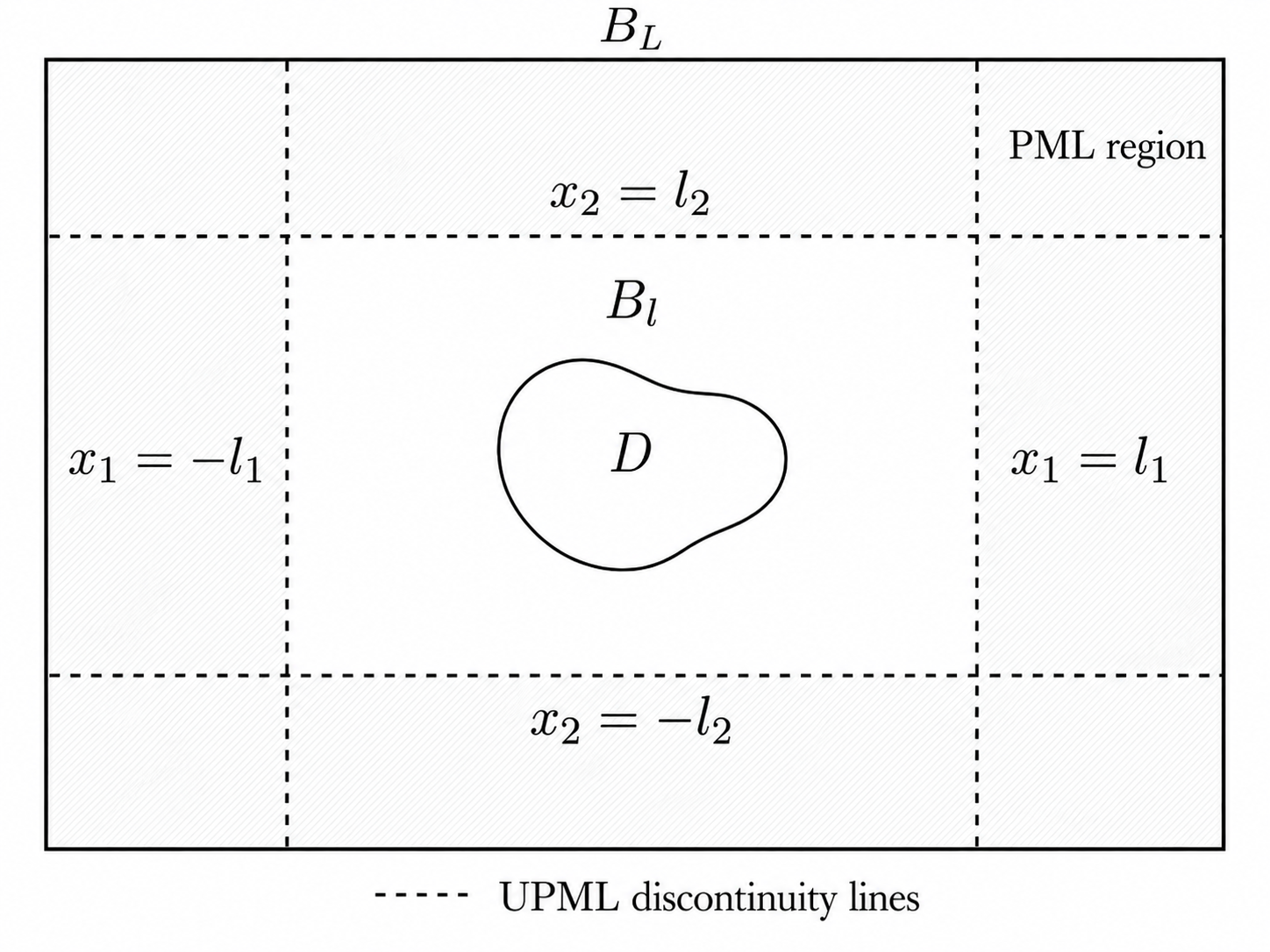}
\caption{Geometric illustration of the UPML and truncated domain.}
\label{fig:upml-geometry}
\end{figure}

The UPML is defined by performing complex stretching to  each coordinate variable independently, namely, $\tilde{\Bx}=(\tilde x_1,\tilde x_2)$, where
\begin{equation}\label{eq:stretch}
	\tilde x_j(x_j):=\int_0^{x_j}\alpha_j(t)\,{\D} t
	=
	x_j+\Vi\int_0^{x_j}\sigma_j(t)\,{\D} t,
	\qquad j=1,2.
\end{equation}
Here $\alpha_j:=1+\Vi\sigma_j$, $j=1,2$, stand for the medium property functions and
\begin{equation}\label{eq:sigma-assump}
\sigma_j(t)=0 \;\; \hbox{if}\;\; |t|\le l_j,\quad
\sigma_j(t) =\sigma_0 \;\; \hbox{if}\;\; |t|>l_j,
	\quad j=1,2,
\end{equation}
for a fixed constant $\sigma_0>0$.
The profile \eqref{eq:sigma-assump} is piecewise constant. We remark that the results can be extended to smoothing profiles where $\sigma_1$ and $\sigma_2$ increase from $0$ to $\sigma_0$ within a narrow neighborhood of $\partial B_l$.
For $\Bx,\By\in\bbR^2$, the complex distance between $\tilde{\Bx}$ and $\tilde{\By}$ is defined by
\begin{equation}\label{eq:rho}
\rho(\tilde{\Bx},\tilde{\By}):=
\big[(\tilde x_1-\tilde y_1)^2+(\tilde x_2-\tilde y_2)^2\big]^{1/2},
\end{equation}
where the square-root branch is chosen so that $\Re(z^{1/2})>0$ for $z\in\bbC\setminus(-\infty,0]$.
An elementary square-root inequality shows that (cf. \cite[Lemma~2.4]{ChenXiangSTDDM_PartI})
\begin{equation}\label{eq:Imrho-lower}
\Im\rho(\tilde{\Bx},\tilde{\By})
\ge \sum_{j=1}^2\frac{x_j-y_j}{|\Bx-\By|}
\int_{y_j}^{x_j}\sigma_j(t)\,{\D} t.
\end{equation}
The analytical continuation of the Green's function to complex arguments is defined by
\begin{equation}\label{eq:tG}
G(\tilde{\Bx},\tilde{\By}) :=\frac{\Vi}{4}H_0^{(1)}\!\bigl(k\,\rho(\tilde{\Bx},\tilde{\By})\bigr),
		\quad \Bx\neq \By.
\end{equation}

\subsection{Exponential decay of $G(\tilde\Bx,\tilde\By)$}

Next we study the asymptotic behavior of $G(\tilde\Bx,\tilde\By)$ as $|\Bx-\By|\to \infty$.
First we prove some preliminary estimates of the complex distance function.

\begin{lemma}\label{lem:rho-cmpr}
For $\Bx,\By\in\mathbb R^2$, there holds
$(1+\sigma_0)^{-1}|\Bx-\By|
\le|\rho(\tilde{\Bx},\tilde{\By})| \le	(1+\sigma_0)|\Bx-\By|$.
\end{lemma}
\begin{proof}
The conclusion is trivial for $\Bx=\By$. Now we assume $\Bx\neq \By$. It is clear that
\begin{align*}
\tilde x_j(x_j)-\tilde y_j(y_j)
=\int_{y_j}^{x_j}[1+\Vi\,\sigma_j(t)]\D t
=\beta_j(x_j-y_j),\quad
\beta_j:= 1+
\frac{\Vi}{x_j-y_j}\int_{y_j}^{x_j}\sigma_j(t)\D t.
\end{align*}
Since $\rho(\tilde{\Bx},\tilde{\By})^2 =
(x_1-y_1)^2\beta_1^2+(x_2-y_2)^2\beta_2^2$, we have
\begin{align*}
|\rho(\tilde{\Bx},\tilde{\By})|^2 =
|\Bx-\By|^2\bigl|t_1\beta_1^2+t_2\beta_2^2\bigr|,\quad
t_j:=\frac{|x_j-y_j|^2}{|\Bx-\By|^2},\quad j=1,2.
\end{align*}
Since $\bigl|t_1\beta_1^2+t_2\beta_2^2\bigr|\le t_1|\beta_1|^2+t_2|\beta_2|^2\le (t_1+t_2)(1+\sigma_0^2)
=1+\sigma_0^2$, this yields the upper bound.

For the lower bound, we write $m_j=\Im\beta_j$ and find
\begin{equation*}
\bigl|t_1\beta_1^2+t_2\beta_2^2\bigr|^2= (1-q)^2+4m^2, \quad
q:=t_1m_1^2+t_2m_2^2, \quad
m:=t_1m_1+t_2m_2.
\end{equation*}
Since $q\le \sigma_0 m$, we have $	\bigl|t_1\beta_1^2+t_2\beta_2^2\bigr|^2
\ge	(1-\sigma_0 m)^2+4m^2$.
	
For $\sigma_0 m\le 1$, the quadratic function
$h(s):=(1-\sigma_0 s)^2+4s^2$ has a minimum on $(0,+\infty)$
\[
h(s_*) =\frac{4}{\sigma_0^2+4}, \quad
s_* :=\frac{\sigma_0}{\sigma_0^2+4}.
\]
This yields $\bigl|t_1\beta_1^2+t_2\beta_2^2\bigr|^2
\ge 4/(\sigma_0^2+4)$. Otherwise, for $\sigma_0 m\ge 1$, we have
\[
\bigl|t_1\beta_1^2+t_2\beta_2^2\bigr|^2
\ge 4m^2 \ge 4/(\sigma_0^2+4)
\ge (1+\sigma_0)^{-2}.
\]
The proof is finished.
\end{proof}

\begin{lemma}\label{lem:Imrho-large-sep}
Suppose $|\Bx-\By|> R$ with $R:=4\sqrt{2} \max\{l_1,l_2\}$. Then
\[
\Im\rho(\tilde{\Bx},\tilde{\By}) \ge \frac14 \sigma_0\,|\Bx-\By|.
\]
\end{lemma}
\begin{proof}
Let $j\in\{1,2\}$ satisfy $|x_j-y_j|=\max\{|x_1-y_1|,|x_2-y_2|\}$.
Then
\[
|x_j-y_j|^2\ge \frac{1}{2}|\Bx-\By|^2 \ge 16l_j^2.
\]
From \eqref{eq:sigma-assump}, we know that
\[
(x_j-y_j)\int_{y_j}^{x_j}\sigma_j(t)\,{\D} t
\ge \sigma_0|x_j-y_j| (|x_j-y_j|-2l_j)
\ge \frac{1}{2}\sigma_0|x_j-y_j|^2
\ge \frac{1}{4}\sigma_0|\Bx-\By|^2.
\]
The proof is finished by inserting this inequality into \eqref{eq:Imrho-lower}.
\end{proof}

The next lemma summarizes the exponential decay of $G(\tilde{\Bx},\tilde{\By})$ and its derivatives. It follows from exponential damping of Hankel functions away from the real axis and Lemma~\ref{lem:rho-cmpr}. The proof is parallel to that of \cite[Lemma~2.5]{ChenXiangSTDDM_PartI}. We only provide the results here.

\begin{lemma}\label{lem:decay-stretchedG}
There is a constant $C$ independent of $k$ such that for all $\Bx,\By\in\bbR^2$ with $\Bx\neq \By$,
\begin{align}
&|\nabla^j_{\Bx} G(\tilde{\Bx},\tilde{\By})|\le
C k^j\big[(k|\Bx-\By|)^{-1/2}+j(k|\Bx-\By|)^{-j}\big]
e^{-\tfrac12 k\Im\rho(\tilde{\Bx},\tilde{\By})},\quad j=0,1,
    \label{eq:dG-decay}\\
&\bigl|\nabla_{\Bx}\nabla_{\By} G(\tilde{\Bx},\tilde{\By})\bigr|
\le C k^2\big[(k|\Bx-\By|)^{-1/2}+(k|\Bx-\By|)^{-2}\big]
e^{-\tfrac12 k\Im\rho(\tilde{\Bx},\tilde{\By})}.
    \label{eq:dxdyG-decay}
\end{align}
\end{lemma}

A key mechanism behind the UPML is that $\Im\rho(\tilde{\Bx},\tilde{\By})$ grows when $\Bx$
moves into the absorbing layer, which yields exponential decay of the stretched Green function. Combining Lemmas~\ref{lem:Imrho-large-sep} and \ref{lem:decay-stretchedG}, for $|\Bx-\By|>R$, we easily get
\begin{align}
&|\nabla^j_{\Bx} G(\tilde{\Bx},\tilde{\By})|\le
C k^j\big[(k|\Bx-\By|)^{-1/2}+j(k|\Bx-\By|)^{-j}\big]
e^{-\tfrac18\sigma_0 k|\Bx-\By|},\quad j=0,1,
    \label{eq:dG-exp}\\
&\bigl|\nabla_{\Bx}\nabla_{\By} G(\tilde{\Bx},\tilde{\By})\bigr|
\le C k^2\big[(k|\Bx-\By|)^{-1/2}+(k|\Bx-\By|)^{-2}\big]
e^{-\tfrac18\sigma_0 k|\Bx-\By|}. \label{eq:dxdyG-exp}
\end{align}

\subsection{The UPML formulations}

In view of the analytical continuation of the Green's function \eqref{eq:tG}, we define the UPML transformation of the scattering solution $u$ as follows
\begin{equation}\label{eq:tu}
\tilde{u}(\Bx) =\int_{\Omega_l}G(\tilde\Bx,\By) f(\By){\D} \By
+\int_\Gamma \frac{\partial u}{\partial\Bn}(\By)\, G(\tilde\Bx,\By) \D s(\By).
\end{equation}
Clearly $\tilde u = u$ in $\Omega_l$ and satisfies a variable-coefficient PDE of the form
\begin{equation}\label{eq:tu-PDE}
\Cl \tilde u = f \quad
\text{in}\;\; D^c,
\end{equation}
where $\Cl w := -J^{-1}\nabla\!\cdot\!\big(\bbA\nabla w \big) - k^2 w$ denotes the stretched Helmholtz operator acting on $w$, and
\begin{equation}\label{eq:A-J}
\bbA(\Bx) := \diag\big(\alpha_2(x_2)/\alpha_1(x_1),\,
\alpha_1(x_1)/\alpha_2(x_2)\big),\quad
J(\Bx) := \alpha_1(x_1)\alpha_2(x_2).
\end{equation}
Since $\bbA$ and $J$ are piecewise constant and have jumps across $\partial B_l$, \eqref{eq:tu-PDE} is understood in the distributional sense.
It holds in the classical sense in each subregion separated by the lines $x_j=\pm l_j$,
together with natural transmission conditions across those interfaces.

Similarly, we define the UPML fundamental solution by
\begin{equation}\label{eq:pml-G-def}
\widetilde G(\Bx,\By):=J(\By)G(\tilde{\Bx},\tilde{\By}).
\end{equation}
It is easy to verify that $\wt{G}$ satisfies (see \cite{ChenXiangSTDDM_PartI})
\begin{equation}\label{eq:pml-fundamental}
\Cl \wt G(\cdot,\By) =\delta_{\By}(\cdot)
\quad \hbox{in}\;\;\bbR^2.
\end{equation}
Notice that \(\widetilde G(\Bx,\By)\) is not symmetric with respect to $\Bx$ and $\By$ due to the factor \(J(\By)\). However, since \(J(\By)=1\) and \(\tilde{\By}=\By\) for \(\By\in B_l\), \eqref{eq:tu} can also be written as
\begin{equation}\label{eq:tu-1}
\tilde{u}(\Bx) =\int_{\Omega_l}\wt G(\Bx,\By) f(\By){\D} \By
+\int_\Gamma \frac{\partial u}{\partial\Bn}(\By)\, \wt G(\Bx,\By) \D s(\By).
\end{equation}
Using equation \eqref{eq:tu} and Lemmas~\ref{lem:Imrho-large-sep} and \ref{lem:decay-stretchedG}, we know that $\tilde{u}(\Bx)$ decays exponentially as $|\Bx|\to +\infty$. This indicates $\tilde u\in H^1(D^c)$.

A weak formulation of \eqref{eq:tu-PDE} is to find $\tilde{u}\in H^1_{\Gamma}(D^c)$ such that
\begin{align} \label{eq:weak-tu}
\mathscr{A}_{D^c}(\tilde{u},v) = (f,v)_{\Omega_l}
\quad \forall\,v\in H_{\Gamma}^1(D^c),
\end{align}
where for a domain $\Omega\subseteq \bbR^2$, $\mathscr{A}_{\Omega}:\Hone\times\Hone\to \bbC$ is a bilinear form defined by
\begin{align} \label{a-R2}
\mathscr{A}_{\Omega}(w,v):=\int_{\Omega}(\bbA\nabla w\cdot\nabla v -k^2Jwv) .
\end{align}

The exponential decay of $\tilde{u}$ inspires us to introduce the approximate problem of \eqref{eq:tu-PDE} with homogeneous condition on the truncation boundary
\begin{align}\label{eq:hu-model}
\Cl\hat u = f \quad \text{in}\;\; \Omega_L, \quad
\hat{u}=0 \quad \hbox{on}\;\; \Gamma\cup\Gamma_L.
\end{align}
A weak formulation of \eqref{eq:hu-model} is to find $\hat{u}\in H_0^1(\Omega_L)$ such that
\begin{align} \label{eq:weak-hu}
\mathscr{A}_{\Omega_L}(\hat{u},v) = (f,v)_{\Omega_l}
\quad \forall\,v\in\zbHone[\Omega_L].
\end{align}
In the next section, we shall first consider pure source problems for which $D^c=\bbR^2$ and $\Omega_L=B_L$. The inf-sup conditions for $\mathscr{A}_{\bbR^2}$ and $\mathscr{A}_{B_L}$ will be established. In section~\ref{sec:obstacle}, we further establish the inf-sup conditions for $\mathscr{A}_{D^c}$ and $\mathscr{A}_{\Omega_L}$
and derive error estimates between $u$ and $\hat{u}$.


\section{Useful estimates for the transformed Green's function}
\label{sec:green}

In this section, we prove some useful estimates concerning the transformed Green's function $G(\tilde\Bx,\tilde\By)$.
The argument is to distinguish two cases: $\Bx,\By\in B_l^c$, and one in $B_l$ and the other in $B_l^c$. The first case is obtained by using the exponential decay of $G(\tilde\Bx,\tilde\By)$, while the other relies on both near-field estimates and exponential decay of $G(\tilde\Bx,\tilde\By)$.

\subsection{Estimates of $G(\tilde\Bx,\tilde\By)$ for $\SN{\Bx-\By}\ge R$}

Let $R$ be the constant in Lemma~\ref{lem:Imrho-large-sep}. The estimate for the large separation $\SN{\Bx-\By}\ge R$ is easy by recalling \eqref{eq:dG-exp} and \eqref{eq:dxdyG-exp}.

\begin{lemma}\label{lem:far-field}
There is a constant $C$ which depends on $R$, $\sigma_0$, but not on $k$, such that, for $j=0,1$,
\begin{align}
&\sup_{\Bx\in\bbR^2}\int_{\{\By\in\bbR^2:\ |\Bx-\By|\ge R\}}
\big|\nabla^j_{\Bx}G(\tilde\Bx,\tilde\By)\big|\,{\D} \By
\le Ck^{j-2}\big(1+\sqrt{kR}\big) e^{-\frac{1}{8}\sigma_0kR},
    \label{eq:G-tail-L1}\\
&\int_{\Bx\in B_l}\int_{\{\By\in\bbR^2:\ |\Bx-\By|\ge R\}}
\big|\nabla_{\Bx}^j G(\tilde\Bx,\tilde\By)\big|^2\,{\D} \By\,{\D} \Bx
\le	Ck^{2j-2}e^{-\frac{1}{4}\sigma_0kR}.  \label{eq:G-tail-L2}
\end{align}
\end{lemma}
\begin{proof}
We only prove the lemma for $j=0$. The proof for $j=1$ is parallel.
Using \eqref{eq:dG-exp} and the polar coordinates centered at $\Bx$, we have
\begin{align*}
\int_{|\Bx-\By|\ge R}\big|G(\tilde\Bx,\tilde\By)\big|\,{\D} \By
\le C\int_R^\infty (kr)^{-1/2}e^{-\frac{1}{8}\sigma_0 kr} r{\D} r
=Ck^{-2}\int_{kR}^\infty e^{-\frac{1}{8}\sigma_0t} t^{1/2}{\D} t.
\end{align*}
Write $t=kR+s$ with $s\ge0$. Then
\[
\int_{kR}^\infty e^{-\frac{1}{8}\sigma_0t} t^{1/2}{\D} t
\le e^{-\frac{1}{8}\sigma_0kR} \int_0^\infty e^{-\frac{1}{8}\sigma_0 s}
\big(\sqrt{s}+\sqrt{kR}\big){\D} s
\le C \big(1+\sqrt{kR}\big) e^{-\frac{1}{8}\sigma_0kR}.
\]
This yields \eqref{eq:G-tail-L1} for $j=0$.
Similarly, we use \eqref{eq:dG-exp} and find that
\begin{align*}
\int_{|\Bx-\By|\ge R}\big|G(\tilde\Bx,\tilde\By)\big|^2{\D} \By
\le Ck^{-1}\int_R^\infty e^{-\frac{1}{4}\sigma_0 kr}{\D} r
\le Ck^{-2}e^{-\frac{1}{4}\sigma_0kR}.
\end{align*}
Finally, since $B_l$ is bounded, the above inequality leads to \eqref{eq:G-tail-L2} for $j=0$.
\end{proof}

\subsection{Estimates of Carleman operators}

Next we present a lemma on the estimates of some Carleman operators. It will be used in the estimate of $G(\tilde\Bx,\tilde\By)$. Write $X:=\Ltwo[0,\infty]$ for convenience.

\begin{lemma}\label{lem:carleman}
Suppose \(\gamma>0\) is a constant and let $K_\gamma(t,s)$ be any of the functions $e^{-\gamma t(t+s)}$, $e^{-\gamma (t-s)^2}$, $e^{-\gamma (t+s)^2}$, and $e^{-\gamma (t^2+s^2)}$. Define the Carleman operator
\begin{align*}
(\Ch_\gamma g)(t) = \int_0^\infty
K_{\gamma}(t,s)g(s)\D s \quad \forall\, g\in X.
\end{align*}
Then there exists a constant $C$ independent of $\gamma$ such that $\|\Ch_\gamma\|_{\Cl(X)}\le C\gamma^{-1/2}$.
\end{lemma}
\begin{proof}
First we consider \(K_\gamma(t,s)=e^{-\gamma t(t+s)}\). Define the unitary scaling
$(\Cu_\gamma g)(s)=\gamma^{-1/4}g(s/\sqrt{\gamma})$. It is clear that $\N{\Cu_\gamma g}_{X} = \N{g}_{X}$.
A change of variables gives $\Cu_\gamma \Ch_\gamma =\gamma^{-1/2}\Cm\Cu_\gamma$, where
\[
(\Cm v)(t) :=\int_0^\infty e^{-t(t+s)}v(s)\,\D s
= e^{-t^2}(\Cm_0 v)(t),\quad
(\Cm_0 v)(t) :=\int_0^\infty e^{-ts}v(s)\,\D s.
\]
It is easy to see that $\N{\Cm}_{\Cl(X)}\le \N{\Cm_0}_{\Cl(X)}$.
We consider the product operator
\begin{align*}
(\Cm_0^*\Cm_0 v)(t) = \int_0^\infty e^{-t s}\int_0^\infty e^{-s s'}v(s')\D s'\D s
=\int_0^\infty \frac{v(s)}{t+s}\D s
:=(\Cc v)(t),
\end{align*}
where $\Cc$ is the integral operator with the kernel $K_0(t,s):= (t+s)^{-1}$. We use the weighted Schur test with $w(s)=s^{-1/2}$ to estimate $\Cc$ \cite{Zhao2015,Zhao2023correction}. It is clear that
\[
\int_0^\infty \frac{w(s)}{t+s}\D s = \int_0^\infty \frac{s^{-1/2}}{t+s}\D s
=t^{-1/2} \int_0^\infty \frac{y^{-1/2}}{1+y}\D y
=\pi w(t).
\]
By the symmetry of $K_0(t,s)$, the second Schur condition holds naturally. This shows
$\|\mathcal C\|_{\Cl(X)}\le \pi$. Therefore, $\N{\Cm}_{\Cl(X)}\le \N{\Cm_0}_{\Cl(X)}\le \sqrt{\pi}$, and
$\|\Ch_\gamma\|_{\Cl(X)}\le \gamma^{-1/2}\N{\Cm}_{\Cl(X)} \le \pi^{1/2}\gamma^{-1/2}$.

For \(K_\gamma(t,s)=e^{-\gamma (t-s)^2}\), we extend \(g\) by zero to \(\mathbb R\) and denote the extension by \(\tilde g\). Writing
\(E_\gamma(t):=e^{-\gamma t^2}\) and using Young's convolution inequality, we have
\[
\|\Ch_\gamma g\|_{X}
\le \|E_\gamma*\tilde g\|_{L^2(\mathbb R)}
\le \|E_\gamma\|_{L^1(\mathbb R)}\|\tilde g\|_{L^2(\mathbb R)}
= C\gamma^{-1/2}\|g\|_{X}.
\]
The proofs for the other two cases of $K(t,s)$ are similar and omitted. The proof is finished.
\end{proof}

\subsection{Dyadic estimates of the transformed Green's kernel}

Let \(R\) be the constant in Lemma~\ref{lem:Imrho-large-sep} and $M :=\lceil 1+\log_2(Rk)\rceil$ be the smallest integer no smaller than $1+ \log_2(Rk)$. Define $r_m:= 2^mk^{-1}$ for $m=0,\cdots, M$.
It is clear that $r_M>r_{M-1}\ge R$. For fixed $\Bx\in\bbR^2$, there exists an $\Bx$-dependent partition of unity
\begin{align}\label{eq:chim}
&\sum_{m=0}^M\chi_m(\Bx,\cdot) =1, \quad
\chi_m(\Bx,\cdot)\ge 0, \quad \chi_m(\Bx,\cdot)\in C^\infty(\bbR^2), \\
&\supp(\chi_0(\Bx,\cdot)) = \ol{B_{r_1}(\Bx)}, \quad
\supp(\chi_M(\Bx,\cdot)) =\bbR^2\backslash B_{r_{M-1}}(\Bx), \notag \\
&\supp(\chi_m(\Bx,\cdot)) = \ol{B_{r_{m+1}}(\Bx)\backslash
B_{r_{m-1}}(\Bx)}, \;\; 0< m< M, \notag
\end{align}
where $B_r(\Bx)$ denotes the open ball whose radius and center are $r$ and $\Bx$, respectively.

\begin{lemma}\label{lem:chi0M}
There exists a constant $C$ independent of $k$ such that, for any $g,v\in\Ltwo[\bbR^2]$,
\begin{align*}
\bigg|\int_{\bbR^2}\int_{\bbR^2}\chi_0(\Bx,\By)
G(\tilde\Bx,\tilde\By)g(\By)v(\Bx)\D\By \D\Bx\bigg|
\le\,& Ck^{-2}\NLtwo[\bbR^2]{g}\NLtwo[\bbR^2]{v}, \\
\bigg|\int_{\bbR^2}\int_{\bbR^2}\chi_M(\Bx,\By)
G(\tilde\Bx,\tilde\By)g(\By)v(\Bx)\D\By \D\Bx\bigg|
\le\,& Ck^{-3/2}e^{-\frac{1}{8}\sigma_0kR} \NLtwo[\bbR^2]{g}\NLtwo[\bbR^2]{v}.
\end{align*}
\end{lemma}
\begin{proof}
Recall from Lemma~\ref{lem:decay-stretchedG} that
$|G(\tilde\Bx,\tilde\By)|\le C (k|\Bx-\By|)^{-\frac{1}{2}}
e^{-\frac{1}{2}k\Im\rho(\tilde\Bx,\tilde\By)}$.
Since $\chi_0(\Bx,\cdot)$ vanishes outside $B_{r_1}(\Bx)$, we have
\begin{align*}
\bigg|\int_{\bbR^2}\int_{\bbR^2}\chi_0(\Bx,\By) G(\tilde\Bx,\tilde\By)
g(\By)v(\Bx)\D\By \D\Bx\bigg|
\le\,& Ck^{-1/2}\NLtwo[\bbR^2]{v}
\bigg(\int_{\bbR^2}\bigg(\int_{B_{r_1}(\Bx)}
|\Bx-\By|^{-1/2}|g(\By)|\D\By\bigg)^2 \D\Bx\bigg)^{1/2} \\
\le\,& Ck^{-1/2}r_1^{1/2}\NLtwo[\bbR^2]{v}
\bigg(\int_{\bbR^2}\int_{B_{r_1}(\Bx)}
|g(\By)|^2\D\By\D\Bx\bigg)^{1/2} \\
=\,& Ck^{-1/2}r_1^{1/2}\NLtwo[\bbR^2]{v}
\bigg(\int_{B_{r_1}({\bf 0})} \int_{\bbR^2}
|g(\Bz+\Bx)|^2\D\Bx\D\Bz\bigg)^{1/2} \\
\le\,& Ck^{-1/2}r_1^{3/2}\NLtwo[\bbR^2]{g}\NLtwo[\bbR^2]{v}.
\end{align*}
Since \(r_1=2/k\), the last bound is \(Ck^{-2}\NLtwo[\bbR^2]{g}
\NLtwo[\bbR^2]{v}\), which proves the first estimate.

Next, using \eqref{eq:dG-exp} and the fact that $\chi_M(\Bx,\cdot)\equiv 0$ in $B_R(\Bx)$, we deduce that
\begin{align*}
&\bigg|\int_{\bbR^2}\int_{\bbR^2}\chi_M(\Bx,\By) G(\tilde\Bx,\tilde\By)
g(\By)v(\Bx)\D\By \D\Bx\bigg| \\
\le\,& Ck^{-1/2} \bigg(\int_{\bbR^2} \int_{\bbR^2\backslash B_{R}(\Bx)}
e^{-\frac{1}{8}\sigma_0k |\Bx-\By|}|g(\By)|^2\D\By \D\Bx\bigg)^{1/2}
\bigg(\int_{\bbR^2} \int_{\bbR^2\backslash B_{R}(\Bx)}
e^{-\frac{1}{8}\sigma_0k |\Bx-\By|}|v(\Bx)|^2\D\By \D\Bx\bigg)^{1/2}.
\end{align*}
Write $\Bz=\By-\Bx$. It is easy to see that
\begin{align*}
\int_{\bbR^2} \int_{\bbR^2\backslash B_{R}(\Bx)}
e^{-\frac{1}{8}\sigma_0k |\Bx-\By|}|g(\By)|^2\D\By \D\Bx
=\,& \int_{\bbR^2\backslash B_{R}({\bf 0})}
e^{-\frac{1}{8}\sigma_0k |\Bz|}
\bigg(\int_{\bbR^2} |g(\Bz+\Bx)|^2\D\Bx\bigg) \D\Bz
\le Ck^{-1}e^{-\frac{1}{8}\sigma_0 k R}\NLtwo[\bbR^2]{g}^2,\\
\int_{\bbR^2} \int_{\bbR^2\backslash B_{R}(\Bx)}
e^{-\frac{1}{8}\sigma_0k |\Bx-\By|}|v(\Bx)|^2\D\By \D\Bx
\le\,&  Ck^{-1}e^{-\frac{1}{8}\sigma_0 k R}\NLtwo[\bbR^2]{v}^2.
\end{align*}
This finishes the proof.
\end{proof}

It remains to estimate the integrals containing $\chi_m$ for $0<m<M$.
First we consider the integral on $B_l^c\times B_l$.
Clearly \(B_l^c\) is the union of four sub-domains
\begin{align}\label{eq:Sj}
S_j^\pm= \big\{\Bz\in B_l^c:\pm z_j>l_j\big\}, \quad j=1,2.
\end{align}

\begin{lemma}\label{lem:chim}
Suppose \(g\in L^2(B_l)\) and \(v\in L^2(B_l^c)\).
There exists a constant \(C\) independent of \(k\) and \(m\) such that,
for \(0<m<M\),
\[
\bigg|\int_{B_l^c}\int_{B_l}\chi_m(\Bx,\By)
G(\tilde\Bx,\tilde\By)g(\By)v(\Bx)\,\D\By \D\Bx\bigg|
\le C2^m k^{-2}\|g\|_{L^2(B_l)}\|v\|_{L^2(B_l^c)} .
\]
\end{lemma}
\begin{proof}
It suffices to consider the domain $\Bx\in S_1^+$. The proofs for other cases are parallel.
Write $s =x_1-l_1$, $t=l_1-y_1$, and $r=|\Bx-\By|$ for convenience. Using \eqref{eq:Imrho-lower}, we have
\begin{align*}
\Im \rho(\tilde\Bx,\tilde\By)
\ge \sigma_0(x_1-y_1)(x_1-l_1)/r = \sigma_0 s (s+t)/r.
\end{align*}
Extend $g$ and $v|_{S_1^+}$ by zero to the exteriors of $B_l$ and $S_1^+$, respectively, and denote the extensions by $\tilde{g}(t,y_2)$ and $\tilde{v}(s,x_2)$, respectively. Using Lemma~\ref{lem:decay-stretchedG}, we deduce that
\begin{align*}
&\bigg|\int_{S_1^+}\int_{B_l}\chi_m(\Bx,\By)
G(\tilde\Bx,\tilde\By)g(\By)v(\Bx)\D\By\D\Bx\bigg| \\
\le\,& C(kr_m)^{-1/2}\int_{S_1^+}\int_{B_l}
\chi_m(\Bx,\By) e^{-\frac{1}{2}k\Im\rho(\tilde\Bx,\tilde\By)}
|g(\By)v(\Bx)|\D\By\D\Bx  \\
\le\,& C(kr_m)^{-1/2}\int_0^{\infty}\int_0^{\infty}
\bigg(\int_{\bbR^2}\chi_m(l_1+s,x_2,l_1-t,y_2)
|\tilde{v}(s,x_2)\tilde{g}(t,y_2)|\D y_2\D x_2\bigg)
e^{-\frac{\sigma_0k}{4r_m}s(s+t)}\D t \D s.
\end{align*}
Since
$\chi_m(l_1+s,x_2,l_1-t,y_2)=0$ for $
(s+t)^2+(x_2-y_2)^2>r_{m+1}^2=4r_m^2$,
the integral in brackets admits
\begin{align} \label{eq:chim-1}
&\int_{\bbR^2}\chi_m(l_1+s,x_2,l_1-t,y_2)
|\tilde{v}(s,x_2)\tilde{g}(t,y_2)|\D y_2\D x_2 \notag \\
\le\,& \bigg[\int_{\bbR^2}\chi_m(l_1+s,x_2,l_1-t,y_2)
|\tilde{v}(s,x_2)|^2\D y_2\D x_2\bigg]^{\frac{1}{2}}
\bigg[\int_{\bbR^2}\chi_m(l_1+s,x_2,l_1-t,y_2)
|\tilde{g}(t,y_2)|^2\D y_2\D x_2\bigg]^{\frac{1}{2}} \notag\\
\le\,& Cr_m\NLtwo[\bbR]{\tilde{v}(s,\cdot)}
\NLtwo[\bbR]{\tilde{g}(t,\cdot)}.
\end{align}
Write $\xi(s):=\NLtwo[\bbR]{\tilde{v}(s,\cdot)}$, $\eta(t):= \NLtwo[\bbR]{\tilde{g}(t,\cdot)}$, and $\gamma=\frac{\sigma_0k}{4r_m}$. Then combining the above two inequalities and using Lemma~\ref{lem:carleman}, we obtain
\begin{align} \label{eq:chim-2}
\bigg|\int_{S_1^+}\int_{B_l}\chi_m(\Bx,\By)
G(\tilde\Bx,\tilde\By)g(\By)v(\Bx)\D\By\D\Bx\bigg|
\le\,& C(r_m/k)^{1/2}\int_0^{\infty}\int_0^{\infty}
\NLtwo[\bbR]{\tilde{v}(s,\cdot)}\NLtwo[\bbR]{\tilde{g}(t,\cdot)}
e^{-\gamma s(s+t)}\D s \D t \notag \\
\le\,& C(r_m/k)^{1/2}\NLtwo[0,\infty]{\xi}
\NLtwo[0,\infty]{\Ch_\gamma\eta} \notag \\
\le\,& Cr_m k^{-1}\NLtwo[S_1^+]{v}\NLtwo[B_l]{g}.
\end{align}
Summing over the four strips \(S_j^\pm\), \(j=1,2\), we get
\[
\bigg|\int_{B_l^c}\int_{B_l}\chi_m(\Bx,\By)
G(\tilde\Bx,\tilde\By)g(\By)v(\Bx)\D\By\D\Bx\bigg|
\le
Cr_mk^{-1}\NLtwo[B_l]{g}\NLtwo[B_l^c]{v}.
\]
Since \(r_m=2^mk^{-1}\), this proves the desired estimate. 
\end{proof}

Next we estimate the integral on $B_l^c\times B_l^c$. The proof is similar to that of Lemma~\ref{lem:chim}.

\begin{lemma}\label{lem:BlcBlc}
Suppose $g,v\in\Ltwo[B_l^c]$ and $\chi_m$ is the cutoff function in \eqref{eq:chim}. There exists a constant $C$ independent of $k$ and $m$ such that
\[
\bigg|\int_{B_l^c}\int_{B_l^c}\chi_m(\Bx,\By)
G(\tilde\Bx,\tilde\By)g(\By)v(\Bx)
\,\D\By\D\Bx\bigg| \le C2^m k^{-2} \|g\|_{L^2(B_l^c)}\|v\|_{L^2(B_l^c)},\quad
0<m<M.
\]
\end{lemma}
\begin{proof}
Let $S_1^\pm$, $S_2^\pm$ be the sub-domains defined in \eqref{eq:Sj}. Again we only need to consider the integral on \(S_1^+\times S_1^+\).
The proofs for other cases are similar.

Write $s = x_1-l_1$ and $t=y_1-l_1$. It is clear that
\[
\Im\rho(\tilde\Bx,\tilde\By) \ge
\frac{(x_1-y_1)}{|\Bx-\By|}
\int_{y_1}^{x_1}\sigma_1(\tau)\,\D\tau
= \frac{\sigma_0}{|\Bx-\By|} (s-t)^2.
\]
Extend $g|_{S_1^+}$ and $v|_{S_1^+}$ by zero to the exterior of $S_1^+$ and denote the extensions by $\tilde{g}(t,y_2)$ and $\tilde{v}(s,x_2)$, respectively.
By Lemma~\ref{lem:decay-stretchedG} and arguments similar to \eqref{eq:chim-1}, we know that
\begin{align*}
&\bigg|\int_{S_1^+}\int_{S_1^+}\chi_m(\Bx,\By) G(\tilde\Bx,\tilde\By)
g(\By)v(\Bx)\,\D\By\D\Bx\bigg| \\
\le\,& C(kr_m)^{-1/2} \int_0^{\infty}\int_0^{\infty}
\bigg(\int_{\bbR^2}\chi_m(l_1+s,x_2,l_1+t,y_2)
|\tilde{v}(s,x_2)\tilde{g}(t,y_2)|\D y_2\D x_2\bigg)
e^{-\frac{\sigma_0k}{4r_m}(s-t)^2}\D t \D s \\
\le\,& Ck^{-1/2}r_m^{1/2} \int_0^{\infty}\int_0^{\infty}
\NLtwo[\bbR]{\tilde{v}(s,\cdot)} \NLtwo[\bbR]{\tilde{g}(t,\cdot)}
e^{-\frac{\sigma_0k}{4r_m}(s-t)^2}\D t \D s.
\end{align*}
Then using Lemma~\ref{lem:carleman} and arguments similar to \eqref{eq:chim-2}, we obtain
\begin{align*}
\bigg|\int_{S_1^+}\int_{S_1^+}\chi_m(\Bx,\By) G(\tilde\Bx,\tilde\By)
g(\By)v(\Bx)\,\D\By\D\Bx\bigg|
\le  C r_m k^{-1}\NLtwo[S_1^+]{v} \NLtwo[S_1^+]{g} .
\end{align*}
Similarly, we have
\[
\bigg|\int_{S_i^\varepsilon}\int_{S_j^\delta}\chi_m(\Bx,\By)
G(\tilde\Bx,\tilde\By)g(\By)v(\Bx)\,\D\By\D\Bx\bigg|
\le Cr_mk^{-1}\NLtwo[S_i^{\varepsilon}]{v}\NLtwo[S_j^{\delta}]{g},\quad 
1\le i,j\le 2,\; \varepsilon,\delta\in \{+,-\}.
\]
Summing over all pairs of sub-domains and using the equality \(r_m=2^mk^{-1}\), we finish the proof.
\end{proof}

\begin{theorem}\label{thm:green-kernel}
There exists a constant $C$ independent of $k$ such that
\[
\bigg|\int_{\bbR^2}\int_{\bbR^2}
G(\tilde\Bx,\tilde\By)g(\By)v(\Bx)\,\D\By\D\Bx\bigg|
\le C k^{-1} \|g\|_{L^2(\bbR^2)}\|v\|_{L^2(\bbR^2)}
\quad \forall\,g,v\in\Ltwo[\bbR^2].
\]
\end{theorem}
\begin{proof}
It is known that the Newtonian potential $w(\Bx):= \int_{B_l} G(\Bx,\By)g(\By)\D\By$ satisfies
\[
\Delta w+k^2w = -g\chi_{B_l} \;\;\; \hbox{in}\;\bbR^2,
\quad \lim_{r\to \infty}\sqrt{r}(\partial_r w-\Vi kw)= 0,
\]
where $\chi_{B_l}$ is the indicator function of $B_l$. By \cite[Lemma~3.5 ]{MelenkSauter2010}, there exists a constant $C>0$ independent of $k$  such that
$\sTN{w}_{H^1(B_l)}
\le C\|g\|_{L^2(B_l)}$. This shows
\begin{align}\label{eq:BlBl}
\bigg|\int_{B_l}\int_{B_l}
G(\tilde\Bx,\tilde\By)g(\By)v(\Bx)\,\D\By\D\Bx\bigg|
\le C k^{-1} \|g\|_{L^2(B_l)}\|v\|_{L^2(B_l)}.
\end{align}

For $M=\lceil 1+\log_2(Rk)\rceil$ and $0<m<M$, we deduce from Lemmas~\ref{lem:chim} and \ref{lem:BlcBlc} that
\[
\bigg|\int_{B_l^c}\int_{\bbR^2}\chi_m(\Bx,\By)
G(\tilde\Bx,\tilde\By)g(\By)v(\Bx)\,\D\By\D\Bx\bigg|
\le C 2^m k^{-2} \|g\|_{L^2(\bbR^2)}\|v\|_{L^2(B_l^c)}.
\]  
Note that $e^{-\frac18\sigma_0kR}\le Ck^{-\frac{1}{2}}$ for a constant $C$ independent of $k\ge 1$. Using Lemma~\ref{lem:chi0M}, we get
\[
\begin{aligned}
\bigg|\int_{B_l^c}\int_{\bbR^2}
G(\tilde\Bx,\tilde\By)g(\By)v(\Bx)\,\D\By\D\Bx\bigg|
&\le Ck^{-2} \sum_{m=0}^{M-1}2^m
\|g\|_{L^2(\bbR^2)}\|v\|_{L^2(B_l^c)} 
\le Ck^{-1}\|g\|_{L^2(\bbR^2)}\|v\|_{L^2(B_l^c)}.
\end{aligned}
\]
By the symmetry
\(G(\tilde\Bx,\tilde\By)=G(\tilde\By,\tilde\Bx)\)
and by interchanging \(g\) and \(v\), the same estimate holds
for the integral over \(B_l\times B_l^c\).
Together with \eqref{eq:BlBl}, this proves the theorem.
\end{proof}


\section{Wavenumber-explicit stability for the truncated source problem}
\label{sec:source}

The purpose of this section is to establish wavenumber-explicit stability of UPML for the pure source problem where $D=\emptyset$ and $\supp(f)\subset B_l$. The weak formulation \eqref{eq:weak-hu} is to find $\hat{u}\in H_0^1(B_L)$ such that 
\begin{align}\label{pml-source}
\mathscr{A}_{B_L}(\hat{u},v) = (f,v)_{B_l}
\quad \forall\,v\in\zbHone[B_L].    
\end{align}

\subsection{Inf-sup condition for $\mathscr{A}_{\bbR^2}$}

Now we prove the inf-sup condition for the whole-space bilinear form $\mathscr{A}_{\bbR^2}$. Given a linear functional $F\in \Hone[\bbR^2]'$, we consider the weak problem: find $w\in\Hone[\bbR^2]$ such that
\begin{equation}\label{eq:upml-whole}
\mathscr{A}_{\bbR^2}(w,\psi) =\langle F, \psi\rangle
\quad \forall\,\psi\in H^{1}(\bbR^2).
\end{equation}

\begin{lemma}\label{lem:upml-whole}
Problem \eqref{eq:upml-whole} has a unique solution $w\in H^{1}(\bbR^2)$. Moreover, there exists a constant $C>0$ independent of $k$ such that
\begin{equation}\label{eq:upml-R2-stab}
\sTN{w}_{H^{1}(\bbR^2)}\le C k\sTN{F}_{H^{1}(\bbR^2)'}.
\end{equation}
\end{lemma}
\begin{proof}
It is known that \eqref{eq:upml-whole} has a unique solution under
the condition \eqref{eq:sigma-assump}  (cf. e.g.\ \cite{KimPasciak2010}). It is left to show the $k$-explicit stability.

Note that $\Re\big(\alpha_{1}(x_{1})/\alpha_{2}(x_{2})\big)\ge (1+\sigma_0^2)^{-1}$ and $\Re\big(\alpha_{2}(x_{2})/\alpha_{1}(x_{1})\big)\ge (1+\sigma_0^2)^{-1}$. By the Lax-Milgram lemma, there exists a unique solution $w_0\in H^{1}(\bbR^2)$ to the problem
\begin{equation}\label{eq:chen-3.6}
\int_{\bbR^2}(\bbA\nabla w_0\cdot\nabla\psi
+k^{2} w_0\psi) =\langle F, \psi\rangle
\quad \forall\,\psi\in H^{1}(\bbR^2),
\end{equation}
and it admits 
\begin{align}\label{est-w0}
\sTN{w_0}_{H^{1}(\bbR^2)}\le C\sTN{F}_{H^{1}(\bbR^2)'}.  
\end{align}
Define the Newtonian potential
\begin{equation}\label{eq:chen-3.7}
\Phi(\Bx):= \int_{\bbR^2} \tilde{G}(\Bx,\By)w_0(\By)(1+J^{-1}(\By))\,{\D} \By = \int_{\bbR^2} G(\tilde\Bx,\tilde\By)w_0(\By)(1+J(\By))\,{\D} \By.
\end{equation}
For any $v\in L^{2}(\bbR^2)$,  Theorem~\ref{thm:green-kernel} shows that
\begin{align} \label{eq:Phi-v}
\bigg|\int_{\bbR^2}\Phi(\Bx) v(\Bx)\D\Bx\bigg| 
= \bigg|\int_{\bbR^2} 
\int_{\bbR^2}G(\tilde\Bx,\tilde\By)w_0(\By)(1+J(\By))v(\Bx){\D}\By{\D}\Bx\bigg| 
\le Ck^{-1}\NLtwo[\bbR^2]{v}\NLtwo[\bbR^2]{w_0}.
\end{align}
This implies $\|\Phi\|_{L^2(\bbR^2)}\le C k^{-1}\|w_0\|_{L^2(\bbR^2)}
\le C k^{-2}\sTN{F}_{H^1(\bbR^2)'}$. 
Then by equation \eqref{eq:pml-fundamental}, $\Phi$ satisfies
\begin{equation}\label{eq:Phi}
\mathscr{A}_{\bbR^2}(\Phi,\psi) =\int_{\bbR^2} (1+J) w_0\psi 
\quad \forall\,\psi\in H^{1}(\bbR^2).
\end{equation}
Clearly $w_0+k^2\Phi$ solves \eqref{eq:upml-whole}.
The uniqueness of solution shows $w=w_0+k^2\Phi$. 

Note that the entries of $\bbA$ admit $\Re(\alpha_2/\alpha_1) \ge (1+\sigma_0^2)^{-1}$ 
and $\Re(\alpha_1/\alpha_2) \ge (1+\sigma_0^2)^{-1}$. Taking $\psi=\ol\Phi$ in \eqref{eq:Phi} shows that
\begin{align} \label{eq:dPhi-est}
\SNHone[\bbR^2]{\Phi}^2  
\le C\Re \int_{\bbR^2}\bbA\nabla\Phi\cdot\nabla\ol{\Phi} 
\le C\big(k^2\NLtwo[\bbR^2]{\Phi}^2 +\NLtwo[\bbR^2]{w_0} \NLtwo[\bbR^2]{\Phi}\big)
\le C\NLtwo[\bbR^2]{w_0}^2.
\end{align}
Then \eqref{eq:upml-R2-stab} follows from \eqref{est-w0} and the identity $w=w_0+k^2\Phi$.
\end{proof}

\begin{theorem}\label{thm:infsup-R2}
There exists a constant $\mu_0>0$ independent of $k$ such that
\begin{equation}\label{eq:infsup-R2}
\sup_{0\neq \psi\in H^1(\bbR^2)}
\frac{|\mathscr{A}_{\bbR^2}(\phi,\psi)|}{\sTN{\psi}_{H^1(\bbR^2)}}
\ge \frac{\mu_0}{k}\sTN{\phi}_{H^1(\bbR^2)}
\quad \forall\,\phi\in H^1(\bbR^2).
\end{equation}
\end{theorem}
\begin{proof}
It is easy to see that the bilinear form $\mathscr{A}_{\bbR^2}$ is continuous on $\Hone[\bbR^2]$. For \(\phi\in H^1(\bbR^2)\), we define a linear functional \(F\in H^{1}(\bbR^2)'\) as follows
\[
\langle F,\psi\rangle:=
\mathscr{A}_{\bbR^2}(\phi,\psi)	\quad \forall\,\psi\in H^1(\bbR^2).
\]
By Lemma~\ref{lem:upml-whole}, \(\phi\) is the unique solution to the weak problem
\[
\mathscr{A}_{\bbR^2}(\phi,\psi)=\langle F,\psi\rangle	\quad \forall\,\psi\in H^1(\bbR^2).
\]
The stability estimate in Lemma~\ref{lem:upml-whole} shows that
\begin{align*}
\sTN{\phi}_{H^1(\bbR^2)}\le C k\sTN{F}_{H^{1}(\bbR^2)'}
=C k\sup_{0\neq \psi\in H^1(\bbR^2)}
\frac{|\mathscr{A}_{\bbR^2}(\phi,\psi)|}{\sTN{\psi}_{H^1(\bbR^2)}}.
\end{align*}
The proof is finished by setting $\mu_0 = 1/C$.
\end{proof}
	
\subsection{Inf-sup condition for $\mathscr{A}_{B_L}$}

Next we derive the inf-sup condition for $\mathscr{A}_{B_L}$. We adopt the reflection argument in the proof of \cite[Theorem~5.5]{BramblePasciak2013} and specify the dependence of all constants on $k$.

\begin{theorem}	\label{thm:infsup-BL}
Assume $d:=\min\{d_1, d_2\}>1.2R$ and that $\sigma_0d$ is large enough. There exists a constant \(\mu_1>0\) independent of $d_1$, $d_2$, and $k$ such that
\begin{equation}\label{eq:infsup-BL}
\sup_{0\neq v\in H_0^1(B_L)}
\frac{|\mathscr{A}_{B_L}(\phi,v)|}{\sTN{v}_{H^1(B_L)}}
\ge \frac{\mu_1}{k} \sTN{\phi}_{H^1(B_L)}	\quad \forall\,\phi\in H_0^1(B_L).
\end{equation}
\end{theorem}
\begin{proof}
Given $\phi\in H_0^1(B_L)$, we define a linear functional $F\in H^{-1}(B_L)$ by
\[
\langle F,\psi\rangle := \mathscr{A}_{B_L}(\phi,\psi)	
\quad\forall\,\psi\in H_0^1(B_L).
\]
It suffices to show $\sTN{\phi}_{H^1(B_L)}\le C k\sTN{F}_{H^{-1}(B_L)}$ for a constant $C$ independent of $k$.
		
Next we extend $\phi$ from $B_L$ to the rectangle $\wt B_L:=(-L_1-d_1,\,L_1+d_1)\times(-L_2-d_2,\,L_2+d_2)$ by odd reflections. First we define $\phi_1$ on
$(-L_1-d_1,\,L_1+d_1)\times(-L_2,L_2)$ by
\begin{align*}
\phi_1(\Bx)=
\begin{cases}
-\phi(2L_1-x_1,x_2) & \hbox{if}\;\; \Bx\in (L_1,\,L_1+d_1)\times(-L_2,L_2),\\
\phi(x_1,x_2) & \hbox{if}\;\; \Bx\in [-L_1,L_1]\times(-L_2,L_2),\\
-\phi(-2L_1-x_1,x_2) & \hbox{if}\;\; \Bx\in (-L_1-d_1,\,-L_1)\times(-L_2,L_2).
\end{cases}
\end{align*}
Then we define $\tilde\phi\in H^1(\wt B_L)$ by odd reflection of $\phi_1$ across $x_2=\pm L_2$:
\begin{align*}
\tilde\phi(\Bx)=
\begin{cases}
-\phi_1(x_1,2L_2-x_2) & \hbox{if}\;\;\Bx\in (-L_1-d_1,\,L_1+d_1)\times (L_2,\,L_2+d_2),\\
\phi_1(x_1,x_2)	&  \hbox{if}\;\; \Bx\in (-L_1-d_1,\,L_1+d_1)\times [-L_2,L_2],\\
-\phi_1(x_1,-2L_2-x_2) & \hbox{if}\;\;  \Bx\in (-L_1-d_1,\,L_1+d_1)\times (-L_2-d_2,\,-L_2).
\end{cases}
\end{align*}
It is clear that $\tilde\phi\in\Hone[\wt B_L]$. Define a linear functional $\tilde F\in H^{-1}(\wt B_L)$ by
\begin{align}\label{problem-tF}
\langle \tilde F,\psi\rangle :=	\mathscr{A}_{\wt B_L}(\tilde\phi,\psi) 
\quad	\forall\,\psi\in H_0^1(\wt B_L). 
\end{align}

To estimate the norm of $\tilde F$, we extend $\psi|_{\wt B_L\backslash B_L}$ to $B_L$ by reflections. First we define $\psi_1$ by even reflections across $x_2=\pm L_2$
\begin{align*}
\psi_1(\Bx)=
\begin{cases}
\psi(x_1,2L_2-x_2) & \hbox{if}\;\;\Bx\in (-L_1-d_1,\,L_1+d_1)\times(l_2,\,L_2),\\
0	&  \hbox{if}\;\; \Bx\in (-L_1-d_1,\,L_1+d_1)\times [-l_2,l_2],\\
\psi(x_1,-2L_2-x_2) & \hbox{if}\;\;  \Bx\in(-L_1-d_1,\,L_1+d_1)\times(-L_2,\,-l_2).
\end{cases}
\end{align*}
Clearly $\psi_1$ is continuous across the interfaces.
Next we define $\hat\psi$ by reflections of $\psi_1$ across $x_1=\pm L_1$
\begin{align*}
\hat\psi(\Bx)=
\begin{cases}
\psi(2L_1-x_1,x_2) -\psi_1(2L_1-x_1,x_2) 
    & \hbox{if}\;\; \Bx\in (l_1,\,L_1)\times(-L_2,L_2),\\
-\psi_1(x_1,x_2) & \hbox{if}\;\; \Bx\in [-l_1,l_1]\times(-L_2,L_2),\\
\psi(-2L_1-x_1,x_2) -\psi_1(-2L_1-x_1,x_2) 
    & \hbox{if}\;\; \Bx\in (-L_1,\,-l_1)\times(-L_2,L_2).
\end{cases}
\end{align*}
It is easy to check that $\hat\psi$ is continuous across the two interfaces $x_1=\pm l_1$. Therefore, $\hat\psi\in\Hone[B_L]$.
Since $\bbA$ and $J$ are invariant under the above
reflections in each PML subregion, the definitions of $\tilde\phi$ and $\hat\psi$ indicate 
\begin{align*}
\langle \tilde F,\psi\rangle =\,& \mathscr{A}_{B_L}(\phi,\psi) 
+\int_{L_2\le |x_2|\le L_2+d_2}\int_{-L_1-d_1}^{L_1+d_1}
\big(\bbA\nabla\tilde\phi\cdot\nabla\psi 
-k^2J\tilde\phi\psi\big)\D x_1\D x_2 \\
& +\int_{-L_2}^{L_2}\int_{L_1\le |x_1|\le L_1+d_1}
\big(\bbA\nabla\tilde\phi\cdot\nabla\psi 
-k^2J\tilde\phi\psi\big)\D x_1\D x_2 \\
=\,& \mathscr{A}_{B_L}(\phi,\psi-\psi_1) 
+\int_{-L_2}^{L_2}\int_{L_1\le |x_1|\le L_1+d_1}
\big[\bbA\nabla\phi_1\cdot\nabla(\psi-\psi_1)
-k^2J\phi_1(\psi-\psi_1)\big]\D x_1\D x_2 \\
=\,& \mathscr{A}_{B_L}(\phi,\psi-\psi_1-\hat\psi) .
\end{align*}
The definitions of $\psi_1$ and $\hat\psi$ imply $\psi-\psi_1-\hat\psi =0$ on $\partial B_L$, that is, $\psi-\psi_1-\hat\psi\in\zbHone[B_L]$. Then
\begin{align*}
\big|\langle \tilde F,\psi\rangle\big| =  \big|\langle F,\psi-\psi_1-\hat\psi\rangle\big| 
\le \sTN{F}_{H^{-1}(B_L)} \sTN{\psi-\psi_1-\hat\psi}_{\Hone[B_L]}
\le C\sTN{F}_{H^{-1}(B_L)} \sTN{\psi}_{\Hone[B_L]}.
\end{align*}
We conclude $\sTN{\tilde F}_{H^{-1}(\wt B_L)} \le C\sTN{F}_{H^{-1}(B_L)}$.

By the Hahn-Banach theorem, we may extend $\tilde F$ to a functional
$F_1\in H^{1}(\bbR^2)'$ such that
\[
\sTN{F_1}_{H^{1}(\bbR^2)'}=\sTN{\tilde F}_{H^{-1}(\wt B_L)}
\le C\sTN{F}_{H^{-1}(B_L)}.
\]
By Lemma~\ref{lem:upml-whole}, there exists a unique $w\in H^1(\bbR^2)$ which satisfies
\begin{align}
\mathscr{A}_{\bbR^2}(w,\psi)=\langle F_1,\psi\rangle	
    \quad	\forall\,\psi\in H^1(\bbR^2), \qquad
\sTN{w}_{H^1(\bbR^2)}\le	Ck\sTN{F_1}_{H^{1}(\bbR^2)'}
\le	Ck\sTN{F}_{H^{-1}(B_L)}.    \label{w-est}
\end{align}
Define $z= w-\tilde\phi$. Then \eqref{problem-tF} shows
$\mathscr{A}_{\wt B_L}(z,\psi) = \big\langle F_1,\psi\big\rangle -\big\langle\tilde{F},\psi\big\rangle =0$ for $\psi\in H_0^1(\wt B_L)$.
Therefore,
\[
\Cl z =0 \quad \hbox{in}\;\; \wt B_L. 
\]

Next we define $\wh B_{L}:= (-2L_1+2l_1, 2L_1 -2l_1)\times (-2L_2+2l_2, 2L_2 -2l_2)$, and let $\chi\in C_0^\infty(\bbR^2)$ be a cutoff function which satisfies $\chi\equiv 1$ in $\wh B_L$ and $\chi\equiv 0$ outside $\wt B_L$. Write $\tilde{z} =\chi z$ and extend it by zero to $\bbR^2\backslash\wt B_L$. Using \eqref{eq:pml-fundamental} and integration by parts, we find that, for $\Bx\in B_L$,  
\begin{align*}
J(\Bx)z(\Bx) = \int_{\wt B_L} \Cl_{\By}\wt G(\By,\Bx) \,
J(\By)\tilde{z}(\By)\D\By 
= \int_{\wt B_L}  J(\By)\wt G(\By,\Bx) \,(\Cl\tilde{z})(\By)\D\By 
= -\int_{\wt B_L\backslash\wh B_L} \wt G(\By,\Bx) \zeta(\By) \D\By,
\end{align*}
where $\zeta:= \nabla\cdot(\bbA\nabla\chi)\, z +2(\bbA\nabla\chi)\cdot\nabla z$ admits
\begin{align}\label{zeta-est} 
\NLtwo[\wt B_L]{\zeta}  \le C\TNHone[\wt B_L]{z}
\le C\big(\TNHone[\wt B_L]{w}+\sTN{\phi}_{\Hone[B_L]}\big)  .
\end{align}

Recall the identity $G(\tilde\By,\tilde\Bx) = G(\tilde\Bx,\tilde\By)$ and inequality~\eqref{eq:dG-exp}. The assumption of the theorem shows $|\Bx-\By|>d-\max\{l_1,l_2\}>R$. For $j=0,1$, we have
\begin{align*} 
\int_{\wt B_L\backslash\wh B_L}\int_{B_L} 
\big|\nabla^j_{\Bx} G(\tilde\By,\tilde\Bx)\big|^2 \D\Bx\D\By 
\le\,& Ck^{2j} \int_{\wt B_L\backslash\wh B_L}\int_{B_L}
(k|\Bx-\By|)^{-1}e^{-\frac{1}{4}\sigma_0k|\Bx-\By|} \D\Bx\D\By \\
\le\,& Cdk^{2j-1}
\int_{d-\max\{l_1,l_2\}}^{\infty}
e^{-\frac14\sigma_0kr}\,\D r \\
\le\,& C d k^{2j-2}  e^{-\frac{1}{5}\sigma_0 kd}.
\end{align*}
By \eqref{zeta-est} and the identity $\wt G(\By,\Bx) = J(\Bx)G(\tilde\By,\tilde\Bx)$, we obtain
\begin{align*}
\NLtwo[B_L]{\nabla^jz}^2 
=\,&\int_{B_L}\bigg|\int_{\wt B_L\backslash\wh B_L} 
\nabla^j_{\Bx} G(\tilde\By,\tilde\Bx)
\,\zeta(\By) \D\By\bigg|^2 \D\Bx  \\
\le\,& \NLtwo[\wt B_L\backslash \wh B_L]{\zeta}^2 
\int_{\wt B_L\backslash\wh B_L}\int_{B_L} \big|\nabla^j_{\Bx} 
G(\tilde\By,\tilde\Bx)\big|^2 \D\Bx\D\By \\
\le\,& Cd k^{2j-2}  e^{-\frac{1}{5}\sigma_0 kd}
\Big(\TNHone[\wt{B_L}]{w}^2 +\TNHone[B_L]{\phi}^2\Big),\quad j=0,1.
\end{align*}
We conclude that
\begin{align*}
\TNHone[B_L]{\phi}^2 \le \Big(1+Cd e^{-\frac{1}{5}\sigma_0 kd}\Big) 
\TNHone[\wt{B_L}]{w}^2  
+ C d e^{-\frac{1}{5}\sigma_0 kd}\TNHone[B_L]{\phi}^2.
\end{align*}
The proof is finished by using \eqref{w-est} and setting $\sigma_0d$ large enough such that $C d e^{-\frac{1}{5}\sigma_0 kd}<1/2$.
\end{proof}

\subsection{Well-posedness of \eqref{pml-source}}

We end this section with the well-posedness of problem \eqref{pml-source}. It is a direct result of Theorem~\ref{thm:infsup-BL}. We assume \(d>1.2R\) and that \(\sigma_0d\) is sufficiently large in the rest of the paper.
\vspace{1mm}

\begin{theorem}
Problem \eqref{pml-source} has a unique solution $\hat{u}\in\zbHone[B_L]$. Moreover, there exists a constant $C$ independent of $k$ such that
$\TNHone[B_L]{\hat{u}} \le C\NLtwo[B_l]{f}$.
\end{theorem}
\begin{proof}
The existence and uniqueness of the solution come directly from Theorem~\ref{thm:infsup-BL}. To prove the stability,  we use \eqref{eq:infsup-BL} and find
\begin{align*}
\TNHone[B_L]{\hat{u}} \le Ck\sup_{0\neq v\in H_0^1(B_L)}
\frac{|\mathscr{A}_{B_L}(\hat{u},v)|}{\sTN{v}_{H^1(B_L)}}
= Ck\sup_{0\neq v\in H_0^1(B_L)}
\frac{|(f,v)_{B_l}|}{\sTN{v}_{H^1(B_L)}}
\le C\NLtwo[B_l]{f}.
\end{align*}
The proof is finished.
\end{proof}


\section{Stability and error estimate for the truncated obstacle problem}
\label{sec:obstacle}

In this section, we study the truncated UPML problem \eqref{eq:weak-hu} with a nonempty obstacle $D$. The proofs are based on the wavenumber-explicit analyses in the previous section.

\subsection{Inf-sup condition for the original problem}

Let $B_{l_0}({\bf 0})$ and $B_{R_0}({\bf 0})$ be the balls with centers at the origin and radii $l_0:=\frac{1}{2}\min\{l_1,l_2\}$ and $R_0 = \sqrt{l_1^2+l_2^2}$, respectively. Clearly $B_{l_0}({\bf 0})\subset B_l \subset B_{R_0}({\bf 0})$. We choose $l_0$ large enough such that $D\subset B_{l_0}({\bf 0})$, and write
\[
\Gamma_{l_0}:=\partial B_{l_0}({\bf 0}),\quad
\Omega_{l_0}:= B_{l_0}({\bf 0})\backslash\ol{D}, \quad 
\Gamma_{R_0}:=\partial B_{R_0}({\bf 0}),\quad
\Omega_{R_0}:= B_{R_0}({\bf 0})\backslash\ol{D}.
\]

Following \cite{ChandlerWildeMonk2008}, we define a bilinear form $b:H^1_{\Gamma}(\Omega_{R_0})\times H^1_{\Gamma}(\Omega_{R_0}) \to \bbC$ as follows
\begin{align*}
b(\phi,v):= \int_{\Omega_{R_0}} (\nabla \phi\cdot\nabla v - k^2 \phi v) -\langle T(\phi),v\rangle_{\Gamma_{R_0}},
\end{align*}
where $T: H^{1/2}(\Gamma_{R_0}) \to H^{-1/2}(\Gamma_{R_0})$ is the Dirichlet-to-Neumann operator defined by 
\begin{align} \label{eq:DtN-R0}
\forall\, \xi = \sum_{m\in\bbZ} \xi_m e^{\Vi m\theta} \in H^{1/2}(\Gamma_{R_0}), \quad 
T(\xi) := k \sum_{m\in\bbZ} \frac{H^{(1)\,\prime}_m(kR_0)}{H^{(1)}_m(kR_0)} 
\xi_m e^{\Vi m\theta},\quad \theta\in [0,2\pi].
\end{align}
Here $H^{(1)}_m$ denotes the $m^{\rm th}$-order Hankel function of the first kind.

\begin{lemma}[{\cite[Corollary~3.9]{ChandlerWildeMonk2008}}]\label{lem:infsup-b}   
There exists a constant $\mu_b>0$ independent of $k$ such that
\begin{align}\label{eq:infsup-b}
\sup_{0\ne v\in H^1_{\Gamma}(\Omega_{R_0})} 
\frac{|b(\phi,v)|}{\TNHone[\Omega_{R_0}]{v}} 
\ge \frac{\mu_b}{k} \TNHone[\Omega_{R_0}]{\phi} 
\quad \forall\,\phi\in H^1_{\Gamma}(\Omega_{R_0}).
\end{align}
\end{lemma}

\subsection{Inf-sup condition for $\mathscr{A}_{\Omega_L}$}

Now we first establish the inf-sup condition for $\mathscr{A}_{D^c}$, and then prove the inf-sup condition for $\mathscr{A}_{\Omega_L}$ by arguments similar to the proof of Theorem~\ref{thm:infsup-BL}.

\begin{lemma}\label{lem:est-G-obstacle}
There exists a constant $C>0$ independent of $k$ such that
\begin{align*}
\int_{B_l^c}\int_{\Omega_{l_0}} 
\big|\nabla_{\Bx}^j\wt G(\Bx,\By)\big|^2\D\By\D\Bx 
\le Ck^{2j-2},\quad j=0,1.
\end{align*}
\end{lemma}
\begin{proof}
Recall that $|\Bx-\By|\ge l_0$ and $\wt G(\Bx,\By)=G(\tilde\Bx,\By)$ for $\Bx\in B_l^c$ and $\By\in\Omega_{l_0}$.
Using Lemma~\ref{lem:decay-stretchedG} and inequality \eqref{eq:dG-exp}, we deduce that
\begin{align}\label{est-djG-1}
\int_{B_l^c}\int_{\Omega_{l_0}} 
\big|\nabla_{\Bx}^j\wt G(\Bx,\By)\big|^2\D\By\D\Bx 
\le\,& C k^{2j-1}\int_{\Omega_{l_0}}\int_{B_l^c}
\frac{e^{-k\Im\rho(\tilde\Bx,\By)}}{|\Bx-\By|} \D\Bx\D\By \notag \\
\le\,& C k^{2j-1}\int_{\Omega_{l_0}}\bigg(\int_{B_l^c\cap B_R(\By)}
\frac{e^{-k\Im\rho(\tilde\Bx,\By)}}{|\Bx-\By|} \D\Bx
+\int_R^\infty e^{-\frac{1}{4}\sigma_0 kr}\D r\bigg)\D\By \notag \\
\le\,& C k^{2j-2}e^{-\tfrac14 \sigma_0 kR} + C k^{2j-1}
\int_{\Omega_{l_0}}\int_{B_l^c\cap B_R(\By)}
\frac{e^{-k\Im\rho(\tilde\Bx,\By)}}{|\Bx-\By|} \D\Bx \D\By.
\end{align} 
To estimate the second term on the right-hand side, we define
\begin{align*}
r(\Bx) = \max\{|x_1|-l_1, |x_2|-l_2\},\quad \xi(\Bx,\By) = 
\begin{cases}
|x_1-y_1| & \hbox{if}\;\; |x_1|-l_1 \ge |x_2|-l_2, \\
|x_2-y_2| & \hbox{otherwise.}
\end{cases}
\end{align*}
It is clear that $\xi(\Bx,\By)\ge l_0$ and $|\Bx-\By|\le R$ for $\By\in\Omega_{l_0}$ and $\Bx\in B_l^c\cap B_R(\By)$.
Then \eqref{eq:Imrho-lower} implies that $\Im\rho(\tilde\Bx,\By)\ge R^{-1}l_0\sigma_0 r(\Bx)$. Therefore, we obtain 
\begin{align*}
\int_{\Omega_{l_0}}\int_{B_l^c\cap B_R(\By)}
\frac{e^{-k\Im\rho(\tilde\Bx,\By)}}{|\Bx-\By|} \D\Bx \D\By
\le C\int_{B_l^c\cap B_R({\bf 0})}
e^{-k R^{-1}l_0\sigma_0 r(\Bx)} \D\Bx
\le C\int_0^R e^{-k R^{-1}l_0\sigma_0 t}\D t  
\le \frac{C}{k}.
\end{align*}
The proof is finished by inserting this inequality into the right-hand side of \eqref{est-djG-1}. 
\end{proof}
\vspace{1mm}

\begin{lemma}\label{lem:phi-Dc}
Suppose $\psi\in\Hone[\Omega_l]$. There exists a $\phi\in\Hone[D^c]$ such that $\phi =\psi$ on $\Gamma$ and 
\begin{align} \label{eq:weak-phi-Dc}
\mathscr{A}_{D^c}(\phi,v) = 0 \quad
\forall\, v\in H^1_{\Gamma}(D^c).
\end{align}
Moreover, there exists a constant $C>0$ independent of $k$ such that
\begin{align}\label{eq:phi-est}
\TNHone[D^c]{\phi} \le C \TNHone[\Omega_l]{\psi}
+ Ck \sup_{0\ne v\in\zbHone[\Omega_l]} 
\frac{|(\nabla\psi,\nabla v)_{\Omega_l}
-k^2(\psi,v)_{\Omega_l}|}{\TNHone[\Omega_l]{v}}.
\end{align}
\end{lemma}
\begin{proof}
Classical theory shows that the scattering problem
\begin{align}\label{helmholtz-w}
\Delta w + k^2 w =0 \;\; \hbox{in}\; D^c, \quad
w = \psi \;\; \hbox{on}\; \Gamma, \quad 
\lim_{r\to\infty} \sqrt{r}\big(\partial_r w -\Vi k w\big) =0,
\end{align}
has a unique solution $w\in H^1_{\rm loc}(D^c)$. Moreover, $w$ admits the weak formulation
\begin{align}\label{eq:bwv}
b(w,v) = 0 \quad \forall\, v\in H^1_{\Gamma}(\Omega_{R_0}).
\end{align}

Let $\chi\in C^\infty_0(\bbR^2)$ be a cut-off function satisfying $\chi \ge 0$, $\chi\equiv 0$ in $B_l^c$, and $\chi\equiv 1$ in $D$. We extend $\chi\psi$ by zero to the exterior of $B_l$ and denote the extension by $\psi_\chi$.
Using \eqref{eq:infsup-b} and \eqref{eq:bwv}, we have
\begin{align} \label{w-psi-chi}
\TNHone[\Omega_{R_0}]{w- \psi_\chi}
\le Ck \sup_{0\ne v\in H^1_{\Gamma}(\Omega_{R_0})} 
\frac{|b(w-\psi_\chi,v)|}{\TNHone[\Omega_{R_0}]{v}}
= Ck \sup_{0\ne v\in H^1_{\Gamma}(\Omega_{R_0})} 
\frac{|b(\psi_\chi,v)|}{\TNHone[\Omega_{R_0}]{v}}.
\end{align}
Clearly $\chi v\in\zbHone[\Omega_l]$ for $v\in H^1_{\Gamma}(\Omega_{R_0})$. Proper arrangement of the integrand shows that
\begin{align*}
b(\psi_\chi,v) = \int_{\Omega_l}
(\nabla\psi_\chi \cdot \nabla v -k^2\psi_\chi v) 
=\int_{\Omega_l}(\psi\nabla v -v \nabla\psi)\cdot \nabla\chi
+\langle F,\chi v\rangle,
\end{align*}
where $F\in H^{-1}(\Omega_l)$ is the linear functional defined by
\[
\langle F,\varphi\rangle := 
\int_{\Omega_l}(\nabla\psi\cdot \nabla\varphi 
-k^2\psi\varphi)  \quad \forall\, \varphi\in\zbHone[\Omega_l].
\] 
There exists a constant depending only on $\N{\chi}_{W^{1,\infty}(\Omega_l)}$ such that
\begin{align*}
\SN{b(\psi_\chi,v)}\le \,& \N{\psi\nabla v -v \nabla\psi}_{\BL^1(\Omega_l)}
\SN{\chi}_{W^{1,\infty}(\Omega_l)} 
+C\sTN{F}_{H^{-1}(\Omega_l)}\TNHone[\Omega_l]{\chi v}  \\
\le\,& C k^{-1}\TNHone[\Omega_l]{\psi}  \TNHone[\Omega_l]{v}
+C\sTN{F}_{H^{-1}(\Omega_l)}\TNHone[\Omega_l]{v}.
\end{align*}
Inserting this inequality into \eqref{w-psi-chi} and using the triangle inequality, we obtain
\begin{align}\label{eq:est-w-R0}
\TNHone[\Omega_{R_0}]{w} 
\le C\TNHone[\Omega_l]{\psi} + Ck \sTN{F}_{H^{-1}(\Omega_l)}.
\end{align}

It is known that $w$ admits the integral representation
\begin{align*}
w(\Bx) = \int_{\Gamma}\bigg[\frac{\partial G(\Bx,\By)}{\partial\Bn(\By)}w(\By)
- G(\Bx,\By)\, \frac{\partial w}{\partial\Bn}(\By) \bigg]\D s(\By) \quad \forall\,\Bx \in D^c.
\end{align*}
Similar to \eqref{eq:tu-1} and \eqref{eq:weak-tu}, we define the complex extension of $w$ by
\begin{align} \label{def-tw}
\tilde{w}(\Bx) = \int_{\Gamma}\bigg[\frac{\partial \wt G(\Bx,\By)}{\partial\Bn(\By)}w(\By)
- \wt G(\Bx,\By)\, \frac{\partial w}{\partial\Bn}(\By) \bigg]\D s(\By) \quad \forall\,\Bx \in D^c.
\end{align} 
Let $\chi_1\in C^\infty_0(\bbR^2)$ be another cut-off function which satisfies
$\chi_1 \ge 0$, $\chi_1\equiv 0$ outside $B_{l_0}({\bf 0})$, and $\chi_1\equiv 1$ on $\ol D$. For $\Bx\in B_l^c$, integration by parts shows that
\begin{align*} 
\tilde w(\Bx) =\,& \int_{\Gamma}\bigg[\frac{\partial \wt G(\Bx,\By)}{\partial\Bn(\By)}(\chi_1w)(\By)
- \wt G(\Bx,\By)\, \frac{\partial (\chi_1w)}{\partial\Bn}(\By) \bigg]\D s(\By) \\
=\,& \int_{\Omega_{l_0}}\Big[(\chi_1w)(\By) (\Delta_{\By} +k^2) \wt G(\Bx,\By)
- \wt G(\Bx,\By) (\Delta+k^2)(\chi_1w)(\By) \Big]\D\By \\
=\,& -\int_{\Omega_{l_0}}\wt G(\Bx,\By) g(\By)\D\By,
\end{align*} 
where $g:=\Delta\chi_1w +2\nabla\chi_1\cdot\nabla w$. Set $\phi:=\tilde{w}$. Using Theorem~\ref{thm:green-kernel}, we deduce that
\begin{align*}
\NLtwo[D^c]{\phi}\le  Ck^{-1}\NLtwo[\Omega_{l_0}]{g} \le
Ck^{-1}\TNHone[\Omega_{l_0}]{w}.
\end{align*}
From \eqref{eq:pml-fundamental} and \eqref{def-tw}, we have $\phi=w$ in $\Omega_l$ and
\[
\Cl\phi =0\quad \hbox{in}\; D^c,\qquad 
\phi=\psi\quad  \hbox{on}\; \Gamma.
\]
Moreover, the inequality \eqref{eq:dG-exp} implies $\phi\in\Hone[D^c]$. Multiplying both sides of the above equation by $v\in H^1_\Gamma(D^c)$ and integrating by parts, we obtain \eqref{eq:weak-phi-Dc}.

Taking $v=\phi-\chi_1\psi$ in \eqref{eq:weak-phi-Dc} and using arguments similar to \eqref{eq:dPhi-est}, we obtain
\begin{align*}
\SNHone[D^c]{v}^2  \le\,& C\Re \int_{D^c}\bbA\nabla v\cdot\nabla\bar{v} 
= C\Re\bigg[-\mathscr{A}_{D^c}(\chi_1\psi,\bar{v}) 
+k^2\int_{D^c}J\SN{v}^2\bigg] \\
\le\,& C\SNHone[\Omega_{l_0}]{\chi_1\psi}\SNHone[\Omega_{l_0}]{v}
+Ck^2\big(\NLtwo[\Omega_{l_0}]{\psi}^2+\NLtwo[D^c]{v}^2\big) \\
\le\,& \frac{1}{2}\SNHone[\Omega_{l_0}]{v}^2 
+C\TNHone[\Omega_{l_0}]{\psi}^2 +Ck^2\NLtwo[D^c]{\phi}^2.
\end{align*}
By \eqref{eq:est-w-R0} and the fact that $\phi=w$ in $\Omega_l$, we obtain
\begin{align*} 
\TNHone[D^c]{\phi} \le C\TNHone[\Omega_{l_0}]{\psi} + C\TNHone[D^c]{v}
\le C\TNHone[\Omega_l]{\psi} +C\TNHone[\Omega_l]{w}
\le C\TNHone[\Omega_l]{\psi} + C k\sTN{F}_{H^{-1}(\Omega_l)}.
\end{align*}
The proof is finished.
\end{proof}

\begin{lemma} \label{lem:infsup-Dc}
There exists a constant $\mu_2>0$ independent of $d_1$, $d_2$, and $k$ such that
\begin{align*}
\sup_{0\ne v\in H^1_{\Gamma}(D^c)} \frac{|\mathscr{A}_{D^c}(\phi,v)|}{\TNHone[D^c]{v}}
\ge \frac{\mu_2}{k} \TNHone[D^c]{\phi} \quad \forall\,\phi\in H^1_{\Gamma}(D^c).
\end{align*}
\end{lemma}
\begin{proof}
Given $\phi\in H^1_{\Gamma}(D^c)$, we define the functional $F\in \big(H_{\Gamma}^{1}(D^c)\big)'$ by 
\[
\langle F, v\rangle = (\nabla\bar\phi,\nabla v)_{D^c} + k^2(\bar\phi,v)_{D^c} \quad 
\forall\,v\in H_{\Gamma}^{1}(D^c).
\]
It is clear that $\TNHone[D^c]{\phi}^2 = \langle F, \phi\rangle$ and $\sTN{F}_{H_{\Gamma}^{1}(D^c)'}=\TNHone[D^c]{\phi}$.

By the Hahn-Banach theorem, we can extend $F$ to a functional $F_1\in H^{1}(\bbR^2)'$ such that
\[
F_1=F\;\; \hbox{on}\; H_{\Gamma}^{1}(D^c),\quad
\sTN{F_1}_{H^{1}(\bbR^2)'} =\sTN{F}_{H_{\Gamma}^{1}(D^c)'}
=\TNHone[D^c]{\phi}.
\]
By Lemma~\ref{lem:upml-whole}, there exists a unique $\psi\in H^1(\bbR^2)$ which satisfies
\begin{align}
&\mathscr{A}_{\bbR^2}(\psi,v)=\langle F_1,v\rangle	
    \quad	\forall\,v\in H^1(\bbR^2), \label{eq:phi1} \\
&\TNHone[\bbR^2]{\psi}\le Ck\sTN{F_1}_{H^{1}(\bbR^2)'}
=	Ck\TNHone[D^c]{\phi}.  \label{ieq:phi1}
\end{align}
Moreover, by Lemma~\ref{lem:phi-Dc}, there exists a $w\in H^1(D^c)$ which satisfies $w=\psi$ on $\Gamma$ and
\begin{align*}
&\mathscr{A}_{D^c}(w,v) = 0 \quad \forall\, v\in H^1_{\Gamma}(D^c), \\
&\TNHone[D^c]{w} \le C \TNHone[\Omega_l]{\psi}
+ Ck\sTN{F_1}_{H^{1}(\bbR^2)'} 
\le Ck \TNHone[D^c]{\phi}.
\end{align*}
It is clear that $\phi_1 :=\psi-w\in H^1_\Gamma(D^c)$ and 
\begin{align*}
&\mathscr{A}_{D^c}(\phi_1,v) = \langle F_1,v\rangle =\langle F,v\rangle		 
    \quad \forall\, v\in H^1_{\Gamma}(D^c), \\
&\TNHone[D^c]{\phi_1} \le \TNHone[D^c]{\psi} + \TNHone[D^c]{w}
\le C_0 k \TNHone[D^c]{\phi},
\end{align*}
where $C_0$ is a constant independent of $k$.
Since $\mathscr{A}_{D^c}$ is symmetric, we have
\begin{align*}
\sup_{0\ne v\in H^1_{\Gamma}(D^c)} \frac{|\mathscr{A}_{D^c}(\phi,v)|}{\TNHone[D^c]{v}}
\ge \frac{|\mathscr{A}_{D^c}(\phi,\phi_1)|}{\TNHone[D^c]{\phi_1}}
= \frac{\langle F,\phi\rangle}{\TNHone[D^c]{\phi_1}}
\ge \frac{\TNHone[D^c]{\phi}^2}{C_0 k \TNHone[D^c]{\phi}}
\ge \frac{\TNHone[D^c]{\phi}}{C_0 k}.
\end{align*}
The proof is finished by setting $\mu_2:=1/C_0$.
\end{proof}

\begin{theorem}  \label{thm:infsup-OmL}
There exists a constant $\mu_3>0$ independent of $d_1$, $d_2$, and $k$ such that
\begin{align}\label{eq:infsup-OmL}
\sup_{0\ne v\in H^1_0(\Omega_L)} 
\frac{|\mathscr{A}_{\Omega_L}(\phi,v)|}{\TNHone[\Omega_L]{v}}
\ge \frac{\mu_3}{k} \TNHone[\Omega_L]{\phi} \quad \forall\,\phi\in H^1_0(\Omega_L).
\end{align}  
\end{theorem}
\begin{proof}
Based on Lemma~\ref{lem:infsup-Dc}, the proof is parallel to that of Theorem~\ref{thm:infsup-BL} and is omitted here.
\end{proof}

\subsection{Error estimate of the PML solution}

To end this section, we prove the exponential convergence of $\hat{u}$ to $\tilde{u}$ in $\Omega_L$ as $\bar\sigma:=\sigma_0 \min\{d_1,d_2\}\to +\infty$. The theory also covers the pure source problem with $D=\emptyset$.

\begin{theorem}\label{thm:upml-error}
Let $\tilde{u}$ be the stretched solution in \eqref{eq:tu} and $\hat{u}$ be the solution of the PML problem \eqref{eq:weak-hu}. Suppose $d_j\ge R + l_j$ for $j=1,2$. There exists a constant $C>0$ independent of $d_1$, $d_2$, and $k$ such that
\begin{equation}\label{eq:trunc-err-final-klogk}
\sTN{\tilde u-\hat u}_{H^1(\Omega_L)} 
\le Ck^{2} e^{-\frac{1}{10}k\bar\sigma}\NLtwo[\Omega_l]{f}.
\end{equation}
\end{theorem}
\begin{proof}
We shall use the conventional norm on $H^{1/2}(\Gamma)$:
\begin{align*}
\N{v}_{H^{1/2}(\Gamma)} :=\,& \bigg(\NLtwo[\Gamma]{v}^2 
+\int_{\Gamma}\int_{\Gamma}\frac{|v(\Bx)-v(\By)|^2}{|\Bx-\By|^2}
\D\By \D\Bx\bigg)^{1/2}.
\end{align*}
Its dual space is $H^{-1/2}(\Gamma)$. The trace inequality on $\Hdiv[\Omega_l]$ shows
\begin{align*}
\N{\partial_\Bn u}_{H^{-1/2}(\Gamma)}  
\le\NLtwo[\Omega_l]{\Delta u} + \SNHone[\Omega_l]{u} 
\le k^2\NLtwo[\Omega_l]{u} +\NLtwo[\Omega_l]{f} + \SNHone[\Omega_l]{u} 
\le C k\NLtwo[\Omega_l]{f}.
\end{align*}
Define $Y_L:= \big\{\Bx\in\bbR^2: |x_j| < L_j-l_j/2, \; j=1,2\big\}$. By \eqref{eq:tu-1}, we have
\begin{align*}
|\nabla^j\tilde{u}(\Bx)| \le\,& \NLtwo[\Omega_l]{f}
\big\|\nabla^j\wt G(\Bx,\cdot)\big\|_{\Ltwo[\Omega_l]}
+\N{\partial_\Bn u}_{H^{-1/2}(\Gamma)} 
\big\|\nabla^j\wt G(\Bx,\cdot)\big\|_{H^{1/2}(\Gamma)} \\
\le\,& \NLtwo[\Omega_l]{f}
\big\|\nabla^j\wt G(\Bx,\cdot)\big\|_{L^2(\Omega_l)}
+C k\NLtwo[\Omega_l]{f}
\big\|\nabla^j\wt G(\Bx,\cdot)\big\|_{H^{1}(\Omega_l)}, 
\end{align*}
for any $\Bx\in B_L\backslash Y_L$ and $j=0,1$. Using \eqref{eq:dG-exp}, we deduce that
\begin{align*} 
&\int_{B_L\backslash Y_L}\int_{B_l} \big|\nabla^j_{\Bx} 
\wt G(\Bx,\By)\big|^2 \D\By\D\Bx
\le Ck^{2j-1} \int_{B_l} \int_{B_L\backslash Y_L}
\frac{e^{-\frac{1}{4}\sigma_0k|\Bx-\By|}}{|\Bx-\By|} \D\Bx\D\By 
\le C k^{2j-2}  e^{-\frac{1}{5}k\bar\sigma},\\
&\int_{B_L\backslash Y_L}\int_{B_l} \big|\nabla^j_{\Bx}\nabla_{\By} 
\wt G(\Bx,\By)\big|^2 \D\By\D\Bx
\le C k^{2j}  e^{-\frac{1}{5}k\bar\sigma}.
\end{align*}
We conclude that
$\|\nabla^j\tilde{u}\|_{L^2(B_L\backslash Y_L)} 
\le Ck^{j+1}e^{-\frac{1}{10}k\bar\sigma}\NLtwo[\Omega_l]{f}$ for $j=0,1$.

Let $\chi\in C^\infty(\bbR^2)$ satisfy $\chi\ge 0$, $\chi\equiv 0$ in $\bbR^2\backslash B_L$, and $\chi\equiv 1$ in $ Y_L$. It is clear that $\phi:=\chi\tilde u-\hat u \in\zbHone[\Omega_L]$. Extend $\phi$ by zero to the exterior of $B_L$ and denote the extension still by $\phi$. Then \eqref{eq:weak-tu} and \eqref{eq:weak-hu} show
\begin{align*}
\mathscr{A}_{\Omega_L}(\phi,v) = \int_{B_L\backslash Y_L}
\big(\tilde{u}\bbA\nabla\chi\cdot\nabla v-v\bbA\nabla\tilde{u}\cdot\nabla\chi\big).
\end{align*}
Using the inf-sup condition \eqref{eq:infsup-OmL} and the Cauchy-Schwarz inequality, we obtain
\begin{align*}
\TNHone[\Omega_L]{\phi} \le Ck \sup_{0\ne v\in \zbHone[\Omega_L]}
\frac{|\mathscr{A}_{\Omega_L}(\phi,v)|}{\TNHone[\Omega_L]{v}}
\le C\TNHone[B_L\backslash Y_L]{\tilde u}.
\end{align*}
Applying triangle inequality to the left-hand side yields
\begin{align*}
\TNHone[\Omega_L]{\tilde{u}-\hat{u}} 
\le C \TNHone[B_L\backslash Y_L]{\tilde u}
\le Ck^{2} e^{-\frac{1}{10}k\bar\sigma}\NLtwo[\Omega_l]{f}.
\end{align*}
The proof is finished.
\end{proof}


\section{The CIP finite element method}
\label{sec:CIP-fem}

In this section, we prove a preasymptotic error estimate for the linear CIP-FEM approximation of the PML problem \eqref{eq:weak-hu}. We shall adopt the duality technique developed in \cite{LiWu2019}. The key ingredients are the inf-sup condition \eqref{eq:infsup-OmL} and the piecewise $H^{1+s}$-regularity (for any $0<s<1$) of the solution to the PML problem: for $\xi\in\Ltwo[\Omega_L]$,
\begin{align}\label{eq:problem-Hs}
-\nabla\cdot(\bbA\nabla w) -k^2 J w = \xi \;\;\;\hbox{in}\;\;\Omega_L,\quad 
w=0 \;\;\;\hbox{on}\;\;\partial\Omega_L.
\end{align}

\subsection{Regularity result of the PML problem}
\label{app:sec:H1s-regularity}

First we study the regularity of the solution to the PML problem in $\Omega_L$. Let \(\Cp_L\) be the set of connected components of 
$\Omega_L\backslash \{\Bx: x_1=\pm l_1\;\hbox{or}\; x_2=\pm l_2\}$. Clearly $\Omega_l\in\Cp_L$. The set of vertices of $\Cp_L$ is denoted by $\Cv_L$, which consists of the four corners of $B_l$, the four corners of $B_L$, and the intersection points of $x_j=\pm l_j$ with $\Gamma_L$ for $j=1,2$. The PML coefficients $\bbA$ and $J$ are discontinuous across the interfaces $x_j=\pm l_j$, $j=1,2$, and cause a loss of global regularity for the solution. 

For $t>0$, we define the broken Sobolev space
\[
H^{t}(\Cp_L):=\big\{v\in\Ltwo[\Omega_L]: v|_Q\in H^{t}(Q)\;\; 
    \forall\,Q\in\Cp_L\big\},
\]
which is equipped with the norm and semi-norm
\[
\N{v}_{H^{t}(\Cp_L)} :=\bigg(\sum_{Q\in\Cp_L} \N{v}^2_{H^t(Q)}\bigg)^{1/2}, \quad 
\SN{v}_{H^{t}(\Cp_L)} :=\bigg(\sum_{Q\in\Cp_L} \SN{v}^2_{H^t(Q)}\bigg)^{1/2}.
\]

We record the local pencil property of the PML coefficients \cite{Nicaise1994a}. Near a vertex in $\Cv_L$, the possible corner singularities of an elliptic transmission problem are generated by the homogeneous principal problem. The lower-order term does not influence the degree of singularity of the solution. Therefore, 
it suffices to study the homogeneous principal transmission problem
\begin{align} \label{eq:trans}
-\nabla\cdot(\bbA\nabla w_0) = \xi \;\;\;
    \hbox{in}\;\;\Omega_L, \quad
w_0 =0 \;\;\; \hbox{on}\;\;\partial\Omega_L. 
\end{align}
The regularity of $w_0$ is determined by eigenvalues of local Mellin pencils near vertices in $\Cv_L$. From \cite{Nicaise1994a}, the singular part of $w_0$ near a vertex $\Bp\in\Cv_L$ has the form
\begin{align} \label{eq:ws-mellin}
w_\Bp(r,\theta) = r^\lambda \phi(\lambda;\theta), \quad \lambda\in\bbC,
\end{align}
where $r^\lambda:=\exp(\lambda\ln r)$, \(r=|\Bx-\Bp|>0\), and $\theta$ is the angular variable. The admissible values of \(\lambda\) are precisely eigenvalues of the local
Mellin pencil at \(\Bp\).

\begin{lemma}\label{lem:mellin}
Suppose $\lambda$ is an eigenvalue of the local Mellin pencil of \eqref{eq:trans} at $\Bp\in\Cv_L$. Then $\Re\lambda\notin (0,1)$. 


\end{lemma}
\begin{proof}
We adapt the Nicaise-S\"andig regularity theory for elliptic transmission problem to \eqref{eq:trans} and identify local pencils for UPML coefficients \cite{Nicaise1994a}. 
It suffices to consider the homogeneous equation
\begin{align}\label{eq:UinQ}
\nabla\cdot(\bbA\nabla U)=0 \quad \hbox{in}\;\; 
\Omega_{\Bp} := \bigcup\left\{\ol{Q}: Q\in\Cp_L,\; \Bp\in\partial Q\right\}.
\end{align}
If $\Bp$ is an intersection of two PML interfaces, there are four sectors sharing $\Bp$, denoted by $Q_1,\cdots,Q_4$, respectively (Fig.~\ref{fig:PL-vert}: left). We can assume that the origin of local coordinate frame is at $\Bp$ and the two interfaces separating the sectors are \(x_1=0\) and \(x_2=0\), respectively.
Inside a fixed sector $Q_j$, $\bbA= \operatorname{diag}\left(\alpha_2/\alpha_1,\,
\alpha_1/\alpha_2\right)$ is constant and $\alpha_1,\alpha_2\in\{1,1+i\sigma_0\}$.
We define $\xi_1=\alpha_1 x_1$, $\xi_2=\alpha_2x_2$, and
$\widehat U(\xi_1,\xi_2)=U(x_1,x_2)$ for \(\Bx\in Q_j\).
By the chain rule and \eqref{eq:UinQ}, we have $\Delta_{\xibf}\widehat U=0$ in $Q_j$. The general solution has the form
\[
\wh U(\xibf) = C_j^+(\xi_1+\Vi\xi_2)^\lambda
+ C_j^-(\xi_1-\Vi\xi_2)^\lambda \quad 
\hbox{for all}\; C^\pm_j\in\bbC\;\hbox{and}\; \lambda\in\bbC.
\]
Therefore, restricted to $Q_j$, the solution $U$ is given by
\begin{align}\label{eq:sector-homogeneous-modes}
U(\Bx) = C_j^+(\alpha_1 x_1 +\Vi \alpha_2 x_2)^\lambda
+C_j^-(\alpha_1 x_1-\Vi \alpha_2 x_2)^\lambda
\quad \forall\,\Bx\in Q_j.
\end{align}

\begin{figure}[htp!]
    \centering
    \includegraphics[width=0.2\linewidth]{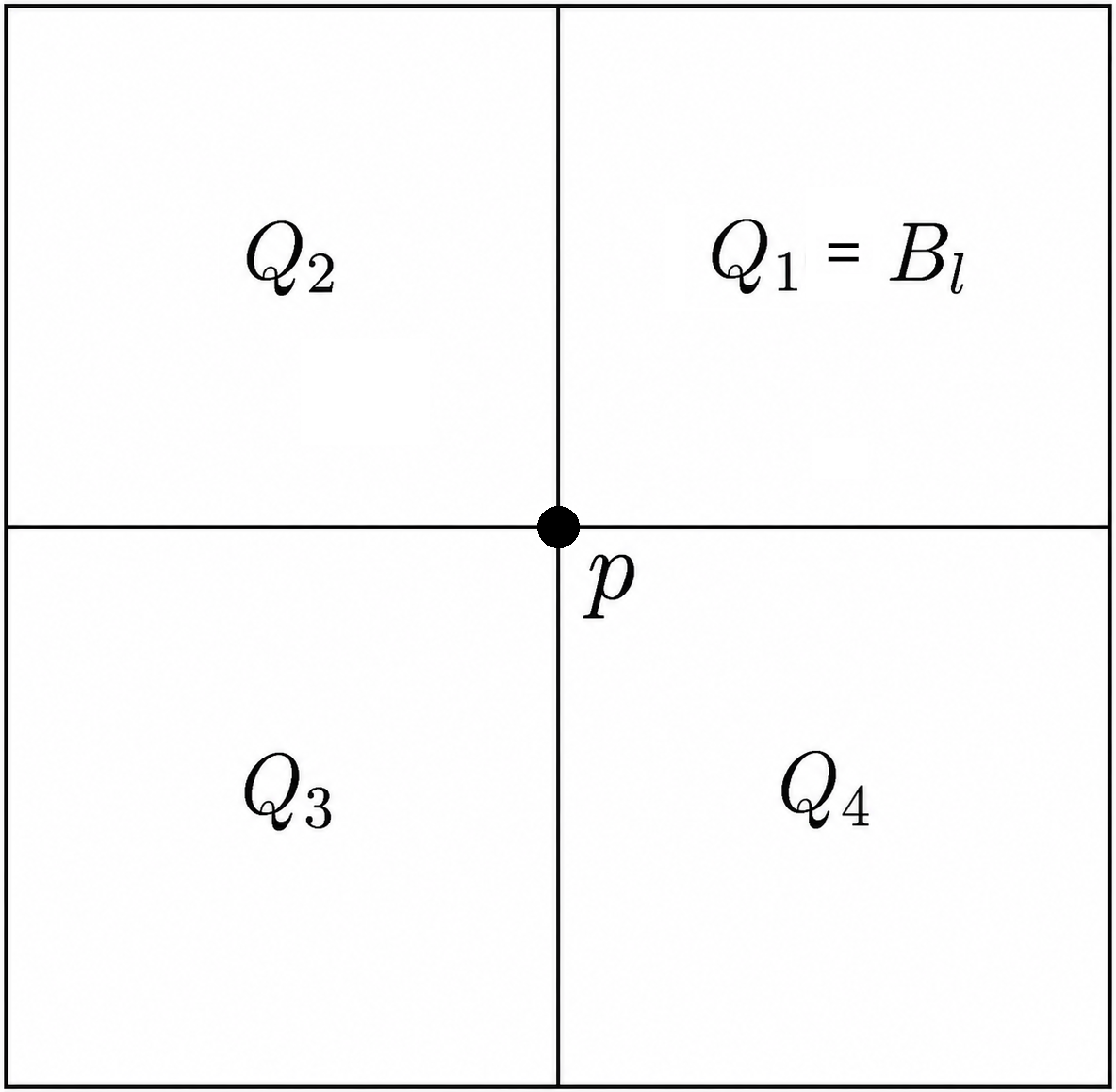}\qquad
    \includegraphics[width=0.2\linewidth]{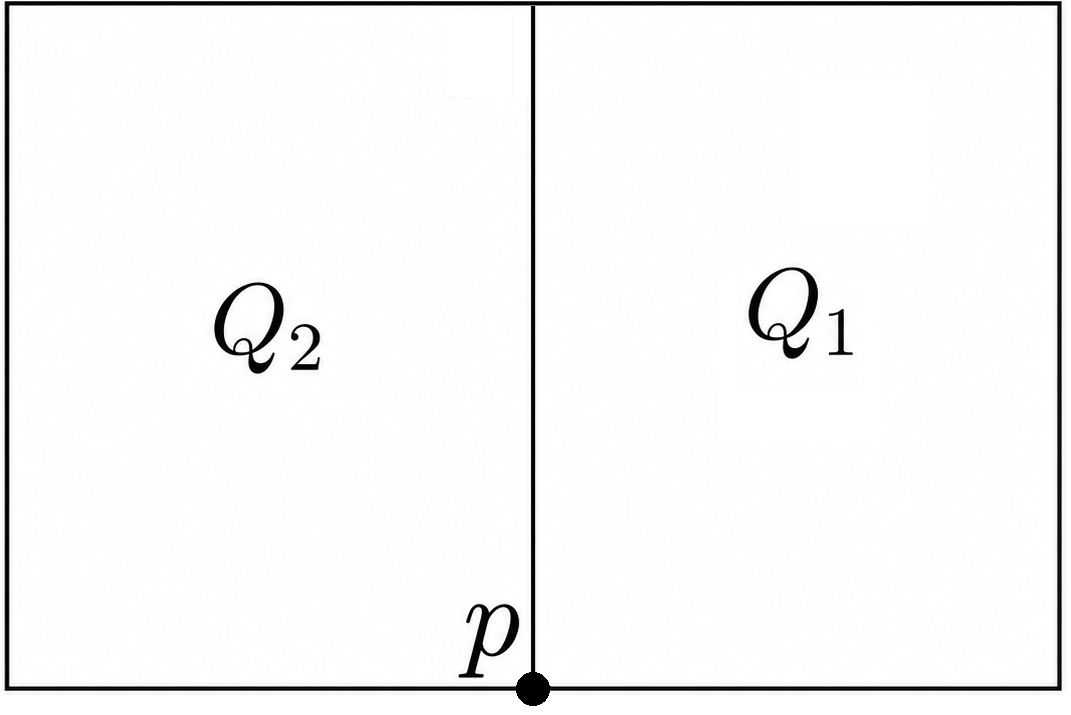}\qquad
    \includegraphics[width=0.155\linewidth]{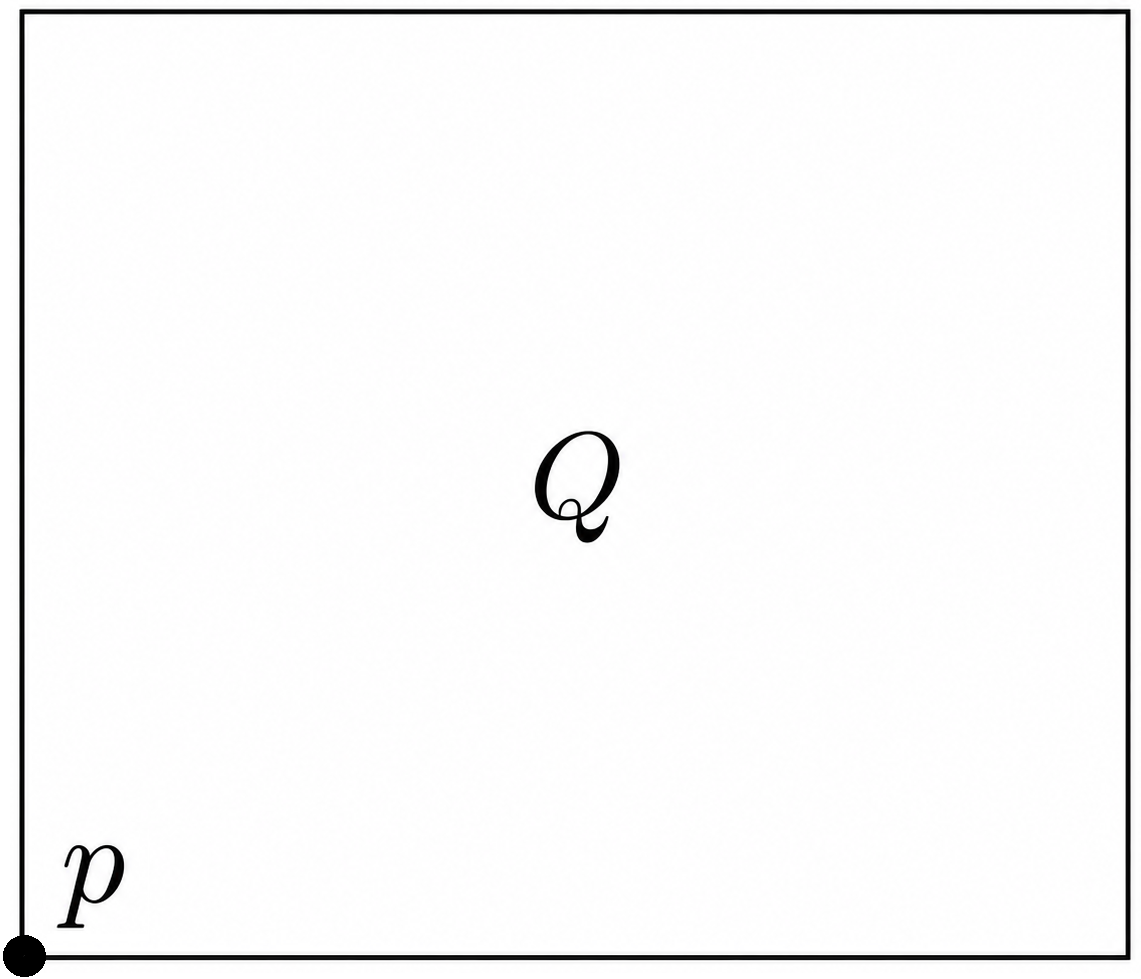}
    \caption{Left: interface-crossing point. Middle: interface-boundary junction. Right: a corner of $B_L$.}
    \label{fig:PL-vert}
\end{figure}

Inside the two sectors \(Q_1\) and \(Q_2\), \(\alpha_2\) is continuous across the interface $\{x_1=0\}$, while \(\alpha_1\) is not. Denote the common value of \(\alpha_2\) by
\(\beta\) and write $\delta^\pm=C_2^\pm-C_1^\pm$.
The continuity of \(U\) across $\{x_1=0\}$ implies
\[
\delta^+(\Vi\beta x_2)^\lambda +\delta^-(-\Vi\beta x_2)^\lambda =0\quad
\hbox{for either $0<|x_2|< l_2$ or $l_2<|x_2|<L_2$}.
\]
Moreover, the continuity of the normal flux $(\bbA\nabla U)\cdot\Bn
=\frac{\alpha_2}{\alpha_1}\frac{\partial U}{\partial x_1}$ across \(\{x_1=0\}\) shows
\[
\delta^+(\Vi\beta x_2)^{\lambda} - \delta^-(-\Vi\beta x_2)^{\lambda} = 0.
\]
This implies
$\delta^+=\delta^-=0$, that is, $C_2^\pm=C_1^\pm$.
The same argument applies to two sectors separated by $\{x_2=0\}$.
Therefore, the expressions of the solution share the same coefficients in the four sectors, denoted by \(C^\pm\). The analyses also apply to the case when $\Bp$ is either an interface-boundary intersection or a corner of $B_L$.

First we consider the case that $\Bp$ is an intersection of two interfaces. By rotating the local coordinate frame, we can assume $\Bp={\bf 0}$, $Q_1 =\Omega_l$, and that $B_l$ is in the first quadrant (see Fig.~\ref{fig:PL-vert}: left). Write $x_1=r\cos\theta$ and $x_2=r\sin\theta$. Since $\alpha_1=\alpha_2=1$ in $\Omega_l$, we can write \eqref{eq:sector-homogeneous-modes} as
\begin{align}\label{eq:Uj}
U(\Bx) = r^\lambda\big(C^+ e^{\Vi \lambda\theta} +C^- e^{-\Vi \lambda\theta}\big) \quad \forall\,\Bx\in \Omega_l. 
\end{align}
Since $U$ is periodic in $\theta$ and smooth inside $\Omega_l$, we have
\[
U|_{\theta = 2\pi+\pi/4} = U|_{\theta =\pi/4},\quad 
\left.\frac{\partial U}{\partial\theta}
\right|_{\theta = 2\pi+\pi/4}
= \left.\frac{\partial U}{\partial\theta}
\right|_{\theta =\pi/4}.
\]
By \eqref{eq:Uj}, the above two conditions are equivalent to
\begin{align*}
\begin{pmatrix}
e^{2\pi\Vi\lambda}-1  &  e^{-2\pi\Vi\lambda}-1 \\
e^{2\pi\Vi\lambda}-1  &  1-e^{-2\pi\Vi\lambda}
\end{pmatrix}
\binom{e^{\Vi\lambda\pi/4}C^+}{e^{-\Vi\lambda\pi/4}C^-}=0 .
\end{align*}
Nonzero solutions of the system require 
$(e^{2\pi\Vi\lambda}-1)(e^{-2\pi\Vi\lambda}-1) =0$.
Therefore, $\Re\lambda\notin (0,1)$.

Next we consider the case in which $\Bp$ is an interface-boundary intersection, that is, $\Bp$ is shared by two sectors $Q_1,Q_2$ in $B_L\backslash B_l$. Again we use the local coordinate frame with $\Bp$ being the origin and $Q_1, Q_2$ in the first and second quadrants, respectively  (Fig.~\ref{fig:PL-vert}: middle). Then  $\alpha_2=1+\Vi\sigma_0$.  Without loss of generality, we can assume $\alpha_1=1+\Vi\sigma_0$ in $Q_1$ and $\alpha_1=1$ in $Q_2$. Similar to \eqref{eq:Uj}, the solution can be written as
\begin{align*}
U(\Bx) = 
\begin{cases}
r^\lambda(1+\Vi\sigma_0)^\lambda  
(C^+e^{\Vi\lambda\theta}
+C^-e^{-\Vi\lambda\theta}) &\hbox{if}\; \Bx\in Q_1, \\
C^+ r_+^\lambda e^{\Vi\lambda\theta_+}
+C^- r_-^\lambda e^{-\Vi\lambda\theta_-} & \hbox{if} \; \Bx\in Q_2,
\end{cases}
\end{align*} 
where $\Bx=r(\cos\theta,\sin\theta)$, $r_\pm
:=\big[(x_1\mp\sigma_0x_2)^2+x_2^2\big]^{1/2}$, 
and $\theta_\pm:= \arccos[(x_1\mp\sigma_0x_2)/r_\pm]$. It is clear that 
\[
\theta(\Bx) =0\quad \forall\,\Bx\in\partial Q_1\cap\{x_2=0\},\qquad 
r_\pm(\Bx) = r, \;\; \theta_\pm(\Bx) =\pi \quad \forall\, \Bx\in\partial Q_2\cap\{x_2=0\}.
\]
The homogeneous boundary condition of $U$ on $(\partial Q_1\cup\partial Q_2)\cap \{x_2=0\}$ implies
\begin{align*}
\begin{pmatrix}
1  &  1 \\
e^{\Vi\lambda\pi}  &  e^{-\Vi\lambda\pi}
\end{pmatrix}
\binom{C^+}{C^-}=0 .
\end{align*}
The vanishing determinant of the coefficient matrix yields $e^{2\Vi\lambda\pi} =1$.
Then $\Re\lambda\notin (0,1)$ in this case.

Finally, if $\Bp$ is a corner of $B_L$, it belongs to only one sector $Q$ (Fig.~\ref{fig:PL-vert}: right). Using \eqref{eq:UinQ}, we have $U(\Bx)= r^\lambda (C^+e^{\Vi\lambda\theta} +C^- e^{-\Vi\lambda\theta})$ with $0< \theta< \pi/2$. The homogeneous boundary condition of $U$ on the two edges connecting $\Bp$ shows
\[
C^+ +C^-=0,\quad
C^+e^{\Vi\lambda\pi/2} +C^- e^{-\Vi\lambda\pi/2} =0.
\]
The vanishing determinant of the coefficient matrix yields $e^{\Vi\lambda\pi} =1$. We still have $\Re\lambda\notin (0,1)$.
\end{proof}

\begin{lemma} \label{lem:regularity-w}
Suppose $w_0$ is the solution to problem \eqref{eq:trans}. 
For any \(0<s<1\), there exists a constant $C_s$ which depends only on $s$ and $\Omega_L$ such that
\begin{equation}\label{app:eq:H1s-shift}
\|w_0\|_{H^{1+s}(\Cp_L)}\le C_s \|\xi\|_{L^2(\Omega_L)}.
\end{equation} 
\end{lemma}
\begin{proof}
Since \(\Gamma\) is \(C^2\)-smooth by the assumption in \eqref{model}, it produces no Mellin singularities. The only conical points are vertices in \(\Cv_L\). 

Define $p= 2/(2-s)$. By Lemma~\ref{lem:mellin}, the local Mellin pencil at any $\Bp\in\Cv_L$ has no eigenvalues satisfying $\Re\lambda = 2-2/p=s$. Therefore, by \cite[Theorem~4.2]{Nicaise1994a}, the solution $w_0$ of \eqref{eq:trans} is regular inside each $Q\in\Cp_L$, and there exists a constant $C_s $ depending only on $Q$ and $s$ such that
\[
\N{w_0}_{W^{2,p}(Q)} \le C_s  \N{\xi}_{L^p(\Omega_L)}
\le C_s  \NLtwo[\Omega_L]{\xi}.
\]
The proof is finished by the Sobolev embedding $W^{2,p}(Q)\hookrightarrow H^{1+s}(Q)$ in two dimensions.
\end{proof}

\begin{theorem}\label{thm:H1s-stability}
Problem \eqref{eq:problem-Hs} has a unique solution. 
For any $s\in (0,1)$, there exists a constant $C_s$ which depends on $s$ and $\Omega_L$, but not on $k$, such that
\begin{equation}\label{eq:eq:H1s-stability}
\TNHone[\Omega_L]{w} + k^{-1}
\N{w}_{H^{1+s}(\Cp_L)} \le C_s \|\xi\|_{L^2(\Omega_L)} .
\end{equation} 
\end{theorem}
\begin{proof}
The inf-sup condition \eqref{eq:infsup-OmL} implies that problem \eqref{eq:problem-Hs} has a unique solution. Moreover, 
\[
\TNHone[\Omega_L]{w} \le  Ck \sup_{0\ne v\in H^1_0(\Omega_L)} 
\frac{|\mathscr{A}_{\Omega_L}(w,v)|}{\TNHone[\Omega_L]{v}}
=Ck\sup_{0\ne v\in H^1_0(\Omega_L)}
\frac{|(\xi,v)_{\Omega_L}|}{\TNHone[\Omega_L]{v}} 
\le C\NLtwo[\Omega_L]{\xi}.
\]
This implies $\NLtwo[\Omega_L]{w}\le Ck^{-1}\NLtwo[\Omega_L]{\xi}$.
From Lemma~\ref{lem:regularity-w}, we know that
\[
\|w\|_{H^{1+s}(\Cp_L)} \le C_s\NLtwo[\Omega_L]{\xi + k^2Jw}
\le C_s\big(\NLtwo[\Omega_L]{\xi} + k^2\NLtwo[\Omega_L]{w}\big)
\le C_s k\NLtwo[\Omega_L]{\xi}.
\]
The proof is finished.
\end{proof}

\subsection{Finite element approximation of \eqref{eq:weak-hu}}
\label{sec:fem}

Suppose \(\widehat K\) is the reference triangle with vertices $(0,0)^T$, $(1,0)^T$, and $(0,1)^T$. Let $\mathcal T_h$ be a triangulation of $\ol\Omega_L$, which is fitted to \(\Gamma\), \(\Gamma_L\), and the UPML interfaces \(x_j=\pm l_j\), $j=1,2$.  More precisely, each element \(K\in\mathcal T_h\) is closed and is the image of \(\widehat K\) under a \(C^2\)-diffeomorphism \(F_K\). The diameter of $K$ is denoted by $h_K$. We assume that there exists a constant $C>0$ independent of $\Ct_h$ such that $h:=\max\limits_{K'\in\Ct_h}h_{K'} \le C h_K$ for all $K\in\Ct_h$ and
\begin{align}\label{eq:DFm}
\|D^mF_K\|_{L^\infty(\widehat K)}\le C h_K^m, \quad
\|D^mF_K^{-1}\|_{L^\infty(K)} \le C h_K^{-m}, \quad
m=1,2.
\end{align}
These assumptions are standard for curved meshes (cf. \cite{Lenoir1986}). It is also standard to assume that the edges of $\Ct_h$ not lying on $\Gamma$ are straight.
The $H^1$-conforming finite element space is defined as
\[
V_h :=\big\{v_h\in H^1_0(\Omega_L):
v_h\circ F_K\in\mathbb P_1(\widehat K)
\quad\forall K\in\mathcal T_h\big\}.
\]

For a function $\hat{v}\in C(\wh{K})$, the linear Lagrange interpolation of $\hat{v}$ on $\wh{K}$ is denoted by $\hat{I}(\hat{v})$. Using pullback mappings, we define the nodal interpolation of a function $v\in C\big(\ol\Omega_L\big)$ by
\[
(I_h v)|_K := \big(\hat{I}(v\circ F_K)\big)\circ F_K^{-1} 
\quad \,\forall\, K\in\Ct_h.
\]
Since all interior edges of $\Ct_h$ are straight, it is easy to verify that $I_hv\in\Hone[\Omega_L]$. The embedding $H^{1+s}(Q)\hookrightarrow C\big(\ol{Q}\big)$ for $s>0$ implies that $I_h$ is well-defined on $H^{1+s}(Q)$ for any $Q\in\Cp_L$. Moreover, using \eqref{eq:DFm} and standard interpolation error estimates on $\wh{K}$, we have
\begin{equation}\label{eq:int-err}
\|v-I_hv\|_{L^2(K)} +h\SNHone[K]{v-I_hv}
\le C h^{1+s}\|v\|_{H^{1+s}(K)} \quad
\forall\,K\in\Ct_h.
\end{equation}

Inspired by Li and Wu \cite{LiWu2019}, we penalize the numerical solution on the set of edges inside $\Omega_l$ 
\[
\Ce_h:=
\big\{E=\partial K^+\cap\partial K^-:
K^\pm\in\Ct_h,\, K^{\pm}\subset\ol\Omega_l,\, K^+\ne K^-\big\}.
\]
For \(E\in\Ce_h\), define the patch $\omega_E:=\hbox{interior}(K^+\cup K^-)$ with $K^\pm\in\Ct_h$ being the two elements sharing \(E\). Across $E$, the jump of normal derivative of a function $v$ is denoted by 
\[
\jump[E]{\partial_\Bn v} :=
\big(\nabla v|_{K^+}\big) \cdot \Bn 
-\big(\nabla v|_{K^-}\big)\cdot \Bn,
\]
where $\Bn$ denotes the unit normal of $E$ pointing from $K^+$ to $K^-$. Define the interior penalty term by
\[
\mathscr{J}_h(u,v):=\sum_{E\in \Ce_h}\gamma_E h_E
\int_E\jump[E]{\partial_{\Bn}u}\,
\jump[E]{\partial_{\Bn}v}.
\]
The penalty parameters $\gamma_E$ satisfy
\begin{equation}\label{eq:gammaE}
|\gamma_E|\le C_\gamma,\quad
\Re\gamma_E\ge -\gamma_0,
\end{equation}
for two positive constants $C_\gamma$ and $\gamma_0$ independent of $h$ and $k$. In practice, $\gamma_0$ is chosen sufficiently small.  
For $0<s \le 1$, we define the bilinear form on $\zbHone[\Omega_L]\cap H^{1+s}(\Cp_L)$ by
\[
\mathscr{A}_h(u,v):= \mathscr{A}_{\Omega_L}(u,v) + \mathscr{J}_h(u,v).
\]

The CIP-finite element approximation to \eqref{eq:weak-hu} is to find \(u_h\in V_h\) such that
\begin{equation}\label{eq:weakh}
\mathscr{A}_h(u_h,v_h)=(f,v_h)_{\Omega_L}
\quad\forall\, v_h\in V_h .
\end{equation}
By \eqref{eq:hu-model}, we have $\Delta\hat{u} = -f-k^2\hat{u}\in \Ltwo[\Omega_l]$. This implies $\jump[E]{\partial_\Bn\hat{u}}=0$ for all $E\in\Ce_h$. Therefore, 
\[
\mathscr{A}_h(\hat{u},v_h) = \mathscr{A}_{\Omega_L}(\hat{u},v_h) =(f,v_h)_{\Omega_L}
\quad\forall\, v_h\in V_h.
\]
This yields the Galerkin orthogonality 
\begin{align}\label{eq:Galerkin-orth}
\mathscr{A}_h(\hat{u}-u_h,v_h) =0
\quad\forall\, v_h\in V_h.
\end{align}

\subsection{Modified elliptic projection}

Following Li and Wu \cite{LiWu2019}, we introduce a modified elliptic projection operator which will play an important role in wavenumber-explicit error estimates. In the UPML case, approximation theory concerning the projection relies on the piecewise $H^{1+s}$-regularity.

Now we define the modified elliptic projection operator $\Pi_h$: $\zbHone[\Omega_L]\to V_h$. Given $\phi\in\zbHone[\Omega_L]$, $\Pi_h \phi\in V_h$ is the solution to the discrete problem
\begin{align}\label{eq:Pih}
\mathscr{B}_h(\Pi_h\phi,v_h)
=\mathscr{B}(\phi,v_h) \quad\forall\,v_h\in V_h,
\end{align}
where the bilinear forms $\mathscr{B}$ and $\mathscr{B}_h$ are, respectively, defined by
\[
\mathscr{B}(p,q) :=
\int_{\Omega_L}\bbA\nabla p\cdot\nabla q, \quad 
\mathscr{B}_h(p,q):=\mathscr{B}(p,q)+\mathscr{J}_h(p,q).
\] 
To prove the coercivity and continuity of $\mathscr{B}_h$, we define the weighted norm and semi-norm
\begin{align*}
\sTN{q}_{\Ct_h}= \big(\TNHone[\Omega_L]{q}^2
+\SN{q}_{\Ce_h}^2\big)^{1/2},\quad 
\SN{q}_{\Ct_h} = \big(\SNHone[\Omega_L]{q}^2
    +\SN{q}_{\Ce_h}^2\big)^{1/2},\quad
\SN{q}_{\Ce_h}^2 = \sum_{E\in \Ce_h}|\gamma_E|h_E
\|\jump[E]{\partial_{\Bn}q}\|_{L^2(E)}^2.  
\end{align*}

\begin{lemma} \label{lem:Bh-stability}
Suppose \(\gamma_0\) is sufficiently small. There exists a positive constant $C_0$ such that
\begin{align}\label{eq:B}
\operatorname{Re}\mathscr{B}(p,\ol{p})
\ge C_0\SNHone[\Omega_L]{p}^2, \quad 
|\mathscr{B}(p,q)| \le C_1
\SNHone[\Omega_L]{p} \SNHone[\Omega_L]{q}
\quad \forall\,p,q\in\Hone[\Omega_L].
\end{align}
Suppose $v_h \in V_h$ and $\SN{p}_{\Ce_h},\SN{q}_{\Ce_h}$ are well-defined, there exists a constant $C_1$ independent of $h$ such that 
\begin{align} \label{eq:Bh}
\operatorname{Re}\mathscr{B}_h(v_h,\ol v_h)
\ge C_0\SNHone[\Omega_L]{v_h}^2, \quad 
|\mathscr{B}_h(p,q)| \le C_1 \SN{p}_{\Ct_h} \SN{q}_{\Ct_h}.
\end{align}
\end{lemma}
\begin{proof}
Since $\Re(\alpha_2/\alpha_1)\ge (1+\sigma_0^2)^{-1}$ and $\Re(\alpha_1/\alpha_2)\ge (1+\sigma_0^2)^{-1}$, \eqref{eq:B} is obvious. Moreover, a direct application of the Cauchy-Schwarz inequality shows the second inequality of \eqref{eq:Bh}.

By \eqref{eq:DFm} and the inverse trace inequality, we have
$h_E\|[\partial_{\Bn}v_h]\|_{L^2(E)}^2\le C\|\nabla v_h\|_{\BL^2(\omega_E)}^2$ for all \(E\in \Ce_h\). Then
\begin{equation}\label{eq:Bh-coer}
\SN{v_h}_{\Ce_h}\le C_1\SNHone[\Omega_L]{v_h}. 
\end{equation}
By \eqref{eq:gammaE}, there exists a constant $C>0$ independent of $h$ such that 
\[
\operatorname{Re}\mathscr{B}_h(v_h,\ol v_h) 
\ge \frac{1}{1+\sigma_0^2}
\SNHone[\Omega_L]{v_h}^2
-\gamma_0 \sum_{E\in \Ce_h}h_E
\|\jump[E]{\partial_{\Bn}v_h}\|_{L^2(E)}^2
\ge \bigg(\frac{1}{1+\sigma_0^2}-C\gamma_0\bigg)
\SNHone[\Omega_L]{v_h}^2.
\]
Choosing \(\gamma_0\) small enough such that $(1+\sigma_0^2)^{-1}\ge 2C\gamma_0$, we obtain the first inequality of \eqref{eq:Bh}.
\end{proof}

\begin{lemma}\label{lem:int-err}
Suppose $0<s\le 1$. There exists a constant $C_s$ depending only on $s$ such that
\begin{align}\label{eq:Jh-Ih}
|I_hv|_{\Ce_h} \le C_s h^s |v|_{H^{1+s}(\Omega_l)}\quad
\forall\,v\in H^{1+s}(\Cp_L).  
\end{align}
Moreover, if $\jump[E]{\partial_\Bn v}=0$ on all $E\in\Ce_h$, then 
\begin{align}\label{eq:Jh-Ih-err}
\SN{v-I_hv}_{\Ct_h} \le C_s h^s|v|_{H^{1+s}(\Cp_L)}.  
\end{align}
\end{lemma}
\begin{proof}
For \(E\in \Ce_h\), let \(K^+\) and \(K^-\) be the two elements sharing \(E\).  By the definition of \(\Ce_h\), $E$ is straight and the patch \(\omega_E=K^+\cup K^-\) is contained in \(\Omega_l\). It is clear that
\[
h_E^{1/2}\NLtwo[E]{\jump[E]{\partial_{\Bn}(I_hv)}} 
=h_E^{1/2}\inf_{q\in\mathbb P_1(\omega_E)}
    \NLtwo[E]{\jump[E]{\partial_{\Bn}(I_h(v -q))}}
\le C\inf_{q\in\mathbb P_1(\omega_E)} \SNHone[\omega_E]{I_h (v -q)}.
\]
Since $v\in H^{1+s}(\Omega_l)$, we deduce from \eqref{eq:int-err} that
\[
h_E^{1/2}\NLtwo[E]{\jump[E]{\partial_{\Bn}(I_hv)}} 
\le C\SNHone[\omega_E]{v-I_hv}
    +C\inf_{q\in\mathbb P_1(\omega_E)}\SNHone[\omega_E]{v-q}
\le C_s\SN{v}_{H^{1+s}(\omega_E)}.
\]
Summing over all $E\in\Ce_h$ yields \eqref{eq:Jh-Ih}.

Finally, if \(\jump[E]{\partial_{\Bn}v}=0\) on all \(E\in\Ce_h\), then
$|v-I_hv|_{\Ce_h}=|I_hv|_{\Ce_h}\le C_sh^s |v|_{H^{1+s}(\Omega_l)}$.
Together with \eqref{eq:int-err}, this shows \eqref{eq:Jh-Ih-err}.
\end{proof}

\begin{theorem}\label{thm:projection}
Problem \eqref{eq:Pih} has a unique solution which satisfies 
\begin{align} \label{eq:Pih-stab}
\SNHone[\Omega_L]{\Pi_h\phi} \le C \SNHone[\Omega_L]{\phi}.
\end{align}
Moreover, if $\phi\in H^{1+s}(\Cp_L)$ with $0<s\le 1$ and
$\jump[E]{\partial_{\Bn}\phi}=0$ for all $E\in \Ce_h$, then there exists a constant $C_s$ which depends only on $s$ and $\Omega_L$ such that 
\begin{equation}\label{eq:error1-proj-H1}
\NLtwo[\Omega_L]{\phi-\Pi_h\phi} +
h^s\SN{\phi-\Pi_h\phi}_{\Ct_h} 
\le C_s h^{2s}|\phi|_{H^{1+s}(\Cp_L)}.
\end{equation}
\end{theorem}
\begin{proof}
Since \(\jump[E]{\partial_{\Bn}\phi}=0\) on all \(E\in \Ce_h\), we have $\mathscr{B}_h(\phi,v_h)=\mathscr{B}(\phi,v_h)$ for all $v_h\in V_h$. This shows
\[
\mathscr{B}_h(\phi-\Pi_h\phi,v_h) =0 \quad
\forall\,v_h\in V_h.
\]
Write $\theta_h:= \Pi_h\phi -I_h\phi$ and $e_h=\phi-I_h\phi$. From \eqref{eq:Bh}, we have
\begin{align*}
\SNHone[\Omega_L]{\theta_h}^2 \le C|\mathscr{B}_h(\theta_h,\bar\theta_h)|
=C|\mathscr{B}_h(e_h,\bar\theta_h)|
\le C\SN{e_h}_{\Ct_h}\SN{\theta_h}_{\Ct_h}
\le C h^s|\phi|_{H^{1+s}(\Cp_L)} \SNHone[\Omega_L]{\theta_h}.
\end{align*}
This yields $\SNHone[\Omega_L]{\theta_h} \le C h^s|\phi|_{H^{1+s}(\Cp_L)}$. Using \eqref{eq:Jh-Ih-err} and \eqref{eq:Bh-coer}, we obtain
\[
\SN{\phi-\Pi_h\phi}_{\Ct_h} 
\le \SN{e_h}_{\Ct_h} + \SN{\theta_h}_{\Ct_h}
\le C h^s|\phi|_{H^{1+s}(\Cp_L)} +\SNHone[\Omega_L]{\theta_h}
\le C h^s |\phi|_{H^{1+s}(\Cp_L)}. 
\]

The $L^2$-norm error is estimated by the duality argument. By \eqref{eq:B}, there is a unique solution $\psi\in\zbHone[\Omega_L]$ to the weak problem
\begin{align}\label{dual-psi}
\mathscr B(\psi,v)
=(\overline{\phi-\Pi_h\phi},v)_{\Omega_L}
\quad\forall v\in H_0^1(\Omega_L).   
\end{align}
Using Lemma~\ref{lem:regularity-w}, there exists a constant $C$ depending on $s$ and $\Omega_L$ such that
\begin{align}\label{stab-psi}
\N{\psi}_{H^{1+s}(\Cp_L)} \le C \NLtwo[\Omega_L]{\phi-\Pi_h\phi}.
\end{align}
Taking $v =\phi-\Pi_h\phi$ in \eqref{dual-psi} and using the symmetry of $\mathscr{B}_h$, we find that
\begin{align*}
\NLtwo[\Omega_L]{\phi-\Pi_h\phi}^2 
=\mathscr{B}(\psi,\phi-\Pi_h\phi) 
=\mathscr{B}_h(\psi -\Pi_h\psi,\phi-\Pi_h\phi)
\le Ch^{2s}\N{\psi}_{H^{1+s}(\Cp_L)}\N{\phi}_{H^{1+s}(\Cp_L)}.
\end{align*}
Finally, from \eqref{stab-psi}, we obtain $\NLtwo[\Omega_L]{\phi-\Pi_h\phi}\le C_s h^{2s}\N{\phi}_{H^{1+s}(\Cp_L)}$.
\end{proof}

\subsection{Preasymptotic finite element error estimates}

Now we are ready to present the main result of this section, that is, the preasymptotic error estimates of the CIP finite element method.  

\begin{theorem}\label{thm:fem-err}
Suppose the parameter \(\gamma_0\) in $\mathscr{J}_h$ is sufficiently small. For any \(s\in (0,1)\), there are two positive constants \(\delta_s\) and \(C_s\) which depend on $s$ and $\Omega_L$, but not on \(h\) and \(k\), such that, when
$k^3h^{2s}\le \delta_s$, the discrete problem \eqref{eq:weakh} has a unique solution
\(u_h\in V_h\), and 
\begin{align*}
&\|\hat u-u_h\|_{L^2(\Omega_L)}
\le C_s (kh^s)^2\|f\|_{L^2(\Omega_l)}, \\
&\NHone[\Omega_L]{\hat u-u_h} \le
C_s(kh^s+k^3h^{2s}) \|f\|_{L^2(\Omega_l)} .    
\end{align*}
\end{theorem}
\begin{proof}
By Theorem~\ref{thm:H1s-stability}, the UPML solution of problem \eqref{eq:weak-hu} admits the regularity estimate
\begin{equation}\label{eq:hu-H1s}
\|\hat{u}\|_{H^{1+s}(\Cp_L)} \le C_s k\|f\|_{L^2(\Omega_l)} .
\end{equation}
Define $e=\hat{u}-u_h$ and consider the dual problem: find $\phi\in\zbHone[\Omega_L]$ such that
\begin{equation}\label{eq:adjoint}
\mathscr{A}_{\Omega_L}(\phi,v)=(\bar{e},v)_{\Omega_L}
        \quad\forall v\in H^1_0(\Omega_L).
\end{equation}
By Theorem~\ref{thm:H1s-stability}, the problem has a unique solution which satisfies
\begin{equation}\label{eq:phi-H1s}
\NHone[\Omega_L]{\phi} + k^{-1}\N{\phi}_{H^{1+s}(\Cp_L)} \le
C_s \|e\|_{L^2(\Omega_L)}.
\end{equation}

Since $\jump[E]{\partial_\Bn\phi}=0$ on all $E\in\Ce_h$, we have $\mathscr{A}_{\Omega_L}(\phi,e) = \mathscr{A}_h(\phi,e)$. Taking $v=e$ in \eqref{eq:adjoint} yields
\[
\|e\|_{L^2(\Omega_L)}^2 = \mathscr{A}_h(\phi,e) = \mathscr{A}_h(\phi-\Pi_h\phi,e)
= \mathscr{B}_h(\phi-\Pi_h\phi,\hat{u}-I_h\hat{u}) 
    -k^2(\phi-\Pi_h\phi, Je)_{\Omega_L},
\]
where $\Pi_h\phi$ is the modified elliptic projection of $\phi$. 
Using \eqref{eq:int-err}, \eqref{eq:Jh-Ih-err}, and Theorem~\ref{thm:projection}, we have
\begin{align*}
\|e\|_{L^2(\Omega_L)}^2 \le Ch^{2s}\N{\phi}_{H^{1+s}(\Cp_L)} 
    \N{\hat{u}}_{H^{1+s}(\Cp_L)}  
    +k^2h^{2s}\N{\phi}_{H^{1+s}(\Cp_L)}\NLtwo[\Omega_L]{e}.
\end{align*}
Together with \eqref{eq:hu-H1s} and \eqref{eq:phi-H1s}, this shows
\begin{align*}
\|e\|_{L^2(\Omega_L)} 
\le C_0 kh^{2s} \N{\hat{u}}_{H^{1+s}(\Cp_L)}  
    +C_1k^3h^{2s}\NLtwo[\Omega_L]{e},
\end{align*}
where $C_0$ and $C_1$ depend on $s$ and $\Omega_L$, but not on $k$ and $h$. Since $k^3h^{2s}\le \delta_s:=(2C_1)^{-1}$ by assumption, we end up with
\begin{align*}
\|e\|_{L^2(\Omega_L)} \le Ckh^{2s}\N{\hat{u}}_{H^{1+s}(\Cp_L)}
\le Ck^2 h^{2s}\NLtwo[\Omega_L]{f}.
\end{align*}

Next we prove the $H^1$-norm error estimate. Write $e=\xi -\eta$ with $\xi:=\hat u-\Pi_h\hat u$ and $\eta:=u_h-\Pi_h\hat u$. By Theorem~\ref{thm:projection} and \eqref{eq:hu-H1s}, we have
\begin{equation}\label{eq:rho-H1}
\NHone[\Omega_L]{\xi} \le  C\bigl(kh^s+k^2h^{2s})\|f\|_{L^2(\Omega_l)} .
\end{equation}
Since \(\eta\in V_h\), we have $\mathscr{B}_h(\xi,\bar\eta)=0$. Therefore,
\[
0=\mathscr{A}_h(e,\bar\eta) = -\mathscr{B}_h(\eta,\bar\eta)-k^2(Je,\bar\eta)_{\Omega_L}.
\]
By \eqref{eq:Bh-coer} and Lemma~\ref{lem:Bh-stability}, we deduce that
\begin{align*}
\SNHone[\Omega_L]{\eta}^2 \le C|\mathscr{B}_h(\eta,\bar\eta)| 
\le Ck^2\|e\|^2_{L^2(\Omega_L)}+ Ck^2\|e\|_{L^2(\Omega_L)}\|\xi\|_{L^2(\Omega_L)} 
\le Ck^6h^{4s} \NLtwo[\Omega_L]{f}^2.
\end{align*}
The proof is finished by combining this inequality with \eqref{eq:rho-H1}.
\end{proof}

\begin{corollary}\label{cor:err}
Suppose $u,u_h$ are the solutions to the scattering problem \eqref{model} and the discrete problem \eqref{eq:weakh}, respectively. Suppose the parameter \(\gamma_0\) is sufficiently small. For any \(s\in (0,1)\), there exist two positive constants \(\delta_s\) and \(C_s\) which depend on $s$ and $\Omega_L$, but not on \(h\) and \(k\), such that, when $k^3h^{2s}\le \delta_s$, 
\begin{align*}
&\|u-u_h\|_{L^2(\Omega_l)}\le C_s \Big(k^2h^{2s}
    +k^{2}e^{-\frac{1}{10}k\bar\sigma}\Big) \|f\|_{L^2(\Omega_l)}, \\
&\NHone[\Omega_l]{u-u_h} \le C_s\Big(kh^s+k^3h^{2s}
    +k^{2} e^{-\frac{1}{10}k\bar\sigma}\Big)
    \|f\|_{L^2(\Omega_l)}.    
\end{align*}
\end{corollary}
\begin{proof}
The results are direct consequences of Theorem~\ref{thm:upml-error} and \ref{thm:fem-err}.
\end{proof}


\section{Numerical experiments}
\label{sec:numerics}

This section reports numerical experiments for solving the truncated UPML problem with both the standard FEM and the CIP-FEM.  The purpose is to illustrate the pollution effect in the high-frequency regime and to show the performance of the CIP-FEM.

\subsection{Reference solution}

We carry out the experiments for the sound-soft scattering problem by the circular obstacle $D=B_a({\bf 0})$ with $a=0.15$.
The incident wave is the plane wave
$u^{\rm inc}(\Bx)=\exp\bigl(\mathrm{i}k\,\Bd\cdot \Bx\bigr)$ with
$\Bd=(\cos\theta_{\rm inc},\sin\theta_{\rm inc})$ and $\theta_{\rm inc}=\pi/6$. The scattered field \(u^s\) satisfies
\begin{equation}\label{eq:num-plane-wave-scattering}
\Delta u^s+k^2u^s=0\;\;\; \text{in } D^c,\quad 
u^s=-u^{\rm inc} \;\;\; \text{on } \Gamma,\quad 
\lim_{r\to\infty}\sqrt{r}(\partial_r u^s-\mathrm{i}k u^s) = 0.
\end{equation}
In polar coordinates, the scattered field admits the series representation
\begin{equation}
\label{eq:num-mie-series}
u^s(r,\theta) = -\frac{J_0(ka)}{H_0^{(1)}(ka)} H_0^{(1)}(kr)
-2 \sum_{n=1}^{\infty}\Vi^n \frac{J_n(ka)}{H_n^{(1)}(ka)}
H_n^{(1)}(kr) \cos(n(\theta-\theta_{\rm inc})),
\quad r>a.
\end{equation}

Now we introduce a smooth radial cut-off function \(\chi=\chi(r)\) satisfying $\chi(r)\equiv 1$ for $r\le 0.22$ and $\chi(r)\equiv 0$ for $r\ge 0.32$, and define $w:=u^s+\chi u^{\rm inc}$. It is clear that \(w\) satisfies  
\begin{equation}
\label{eq:num-transformed-source}
\Delta w +k^2w = -f_\chi :=
(\Delta\chi)u^{\rm inc}+ 2\nabla\chi\cdot\nabla u^{\rm inc}
\;\;\;\hbox{in}\; D^c,\quad 
w=0\;\;\;\hbox{on}\;\Gamma.
\end{equation}
Clearly \(f_\chi\) is supported only in the annulus $0.22<r<0.32$.
In practice, we calculate the reference solution $u^s$ approximately by truncating the series \eqref{eq:num-mie-series} into finite terms $n\le N_{\rm Mie} :=\big\lceil ka+10(ka)^{1/3}+20\big\rceil$.
Moreover, we shall solve \eqref{eq:num-transformed-source} for a numerical solution $w_h$ with the UPML and CIP-FEM methods, and recover the numerical scattered field by $u_h^s=w_h-\chi u^{\rm inc}$.

\subsection{Numerical results}

We take the physical and truncated domains to be $B_l=(-0.5,0.5)^2$ and $B_L=(-1,1)^2$, respectively. The UPML medium parameter is chosen as $\sigma_0=5$. For the CIP-FEM, we use the dispersion-based real penalty parameters in \cite{LiWu2019} for the edges $E\in\Ce_h$, namely, 
\begin{equation}
\label{eq:num-cip-gamma}
\gamma_E\equiv  -\frac{\sqrt3}{24} -
    \frac{\sqrt3}{1728}(k h)^2
    \quad \forall\,E\in\Ce_h.
\end{equation}
The relative approximation error is defined by 
$E_h:=\TNHone[\Omega_l]{w-w_h}/\TNHone[\Omega_l]{w}$, and the relative interpolation error is defined by 
$E_h^{\rm int}:=\TNHone[\Omega_l]{w-I_hw}/\TNHone[\Omega_l]{w}$.

Figure~\ref{fig:obstacle-nodal-errors} shows absolute values of the error function $|w-w_h| =|u^s-u_h^s|$.
The standard FEM yields a pronounced large pollution error in the physical region, while the pollution error produced by the CIP-FEM is substantially smaller and evenly distributed. This indicates that the interior penalty term suppresses the dominant phase error. Figure~\ref{fig:obstacle-trace} compares the real parts of $u^s$ and the numerical solutions on the horizontal line \(x_2=0.35\).
The standard FEM exhibits visible phase and amplitude mismatches, while the numerical solution obtained by the CIP-FEM matches the oscillatory pattern of $u^s$. 

\begin{figure}[ht!]
\centering
\includegraphics[width=0.6\textwidth]{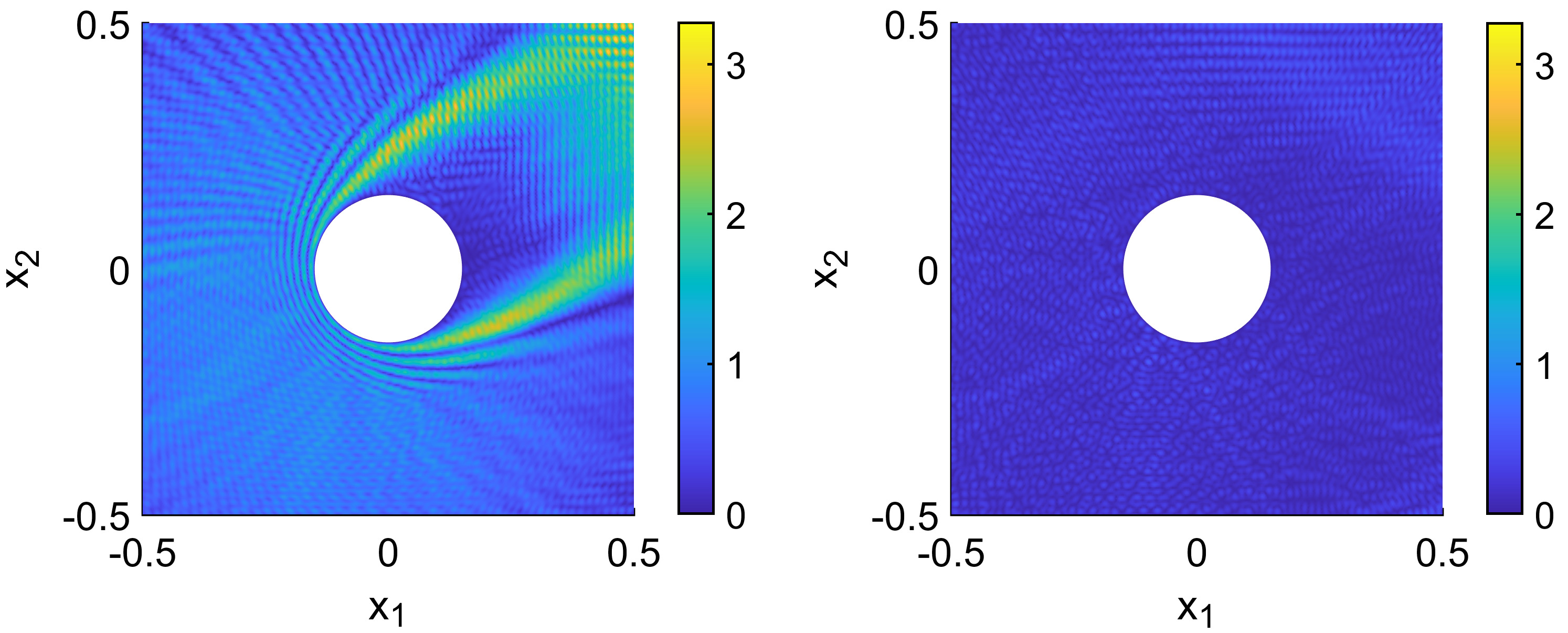}
\caption{Plots of $|w-w_h|$ with $k=240$ and $kh=1$. Left: standard FEM. Right: CIP-FEM.}
\label{fig:obstacle-nodal-errors}
\end{figure}

\begin{figure}[ht!]
\centering
\includegraphics[width=0.5\textwidth]{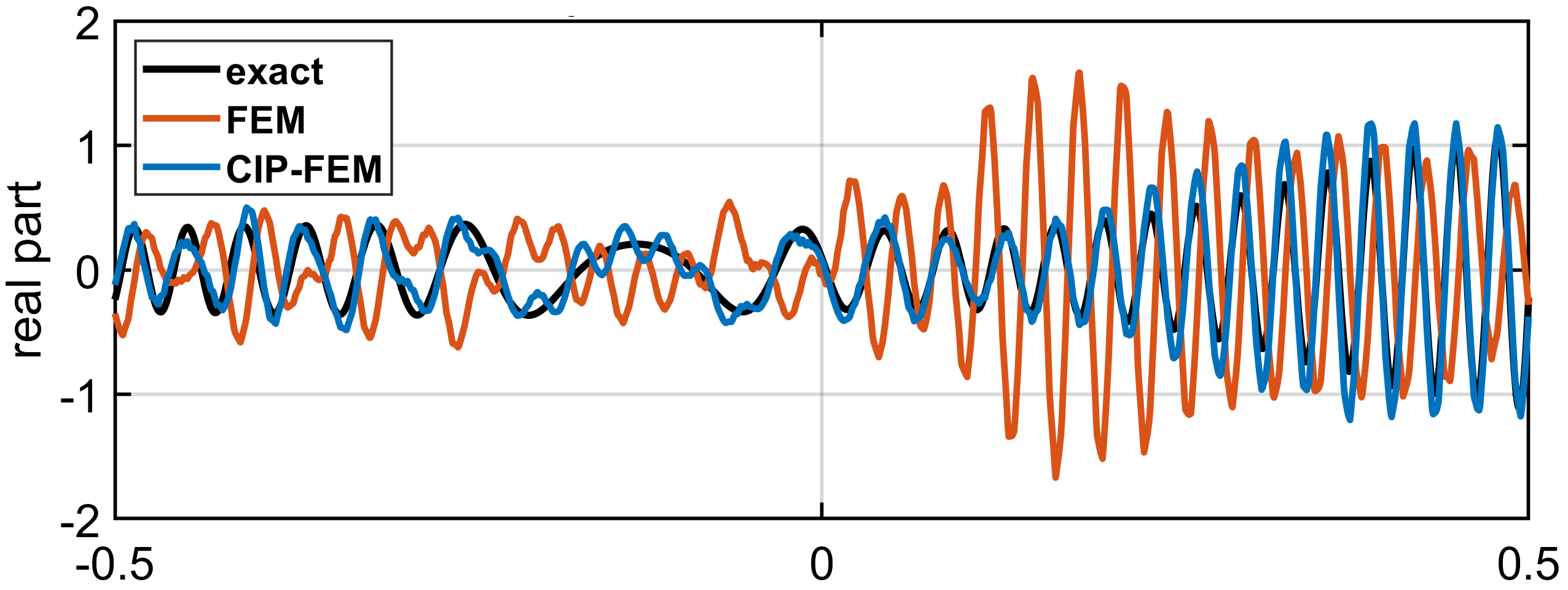}
\caption{Plots of $\Re u^s$ and $\Re u^s_h$ on the line \(\{x_2=0.35\}\) with \(k=240\) and \(kh=1\).}
\label{fig:obstacle-trace}
\end{figure}

Next, we investigate the pollution effects in the numerical solutions. We fix the number of degrees of freedom per wavelength by setting \(kh=1\) and \(0.5\), while increasing the wavenumber $k$.  Figure~\ref{fig:obstacle-fixed-kh-error} plots the interpolation error $E^{\rm int}_h$ and the approximation errors $E_h$ obtained by both FEM and CIP-FEM.
For both \(kh=1\) and $0.5$, the FEM errors grow rapidly as \(k\) increases, and are far above the interpolation error in the large wavenumber regime. However, the CIP-FEM errors remain almost stable to $k$ in both cases. This indicates $E_h = O(kh)$ and that the pollution errors are suppressed remarkably.

\begin{figure}[ht!]
\centering
\includegraphics[width=0.4\textwidth]{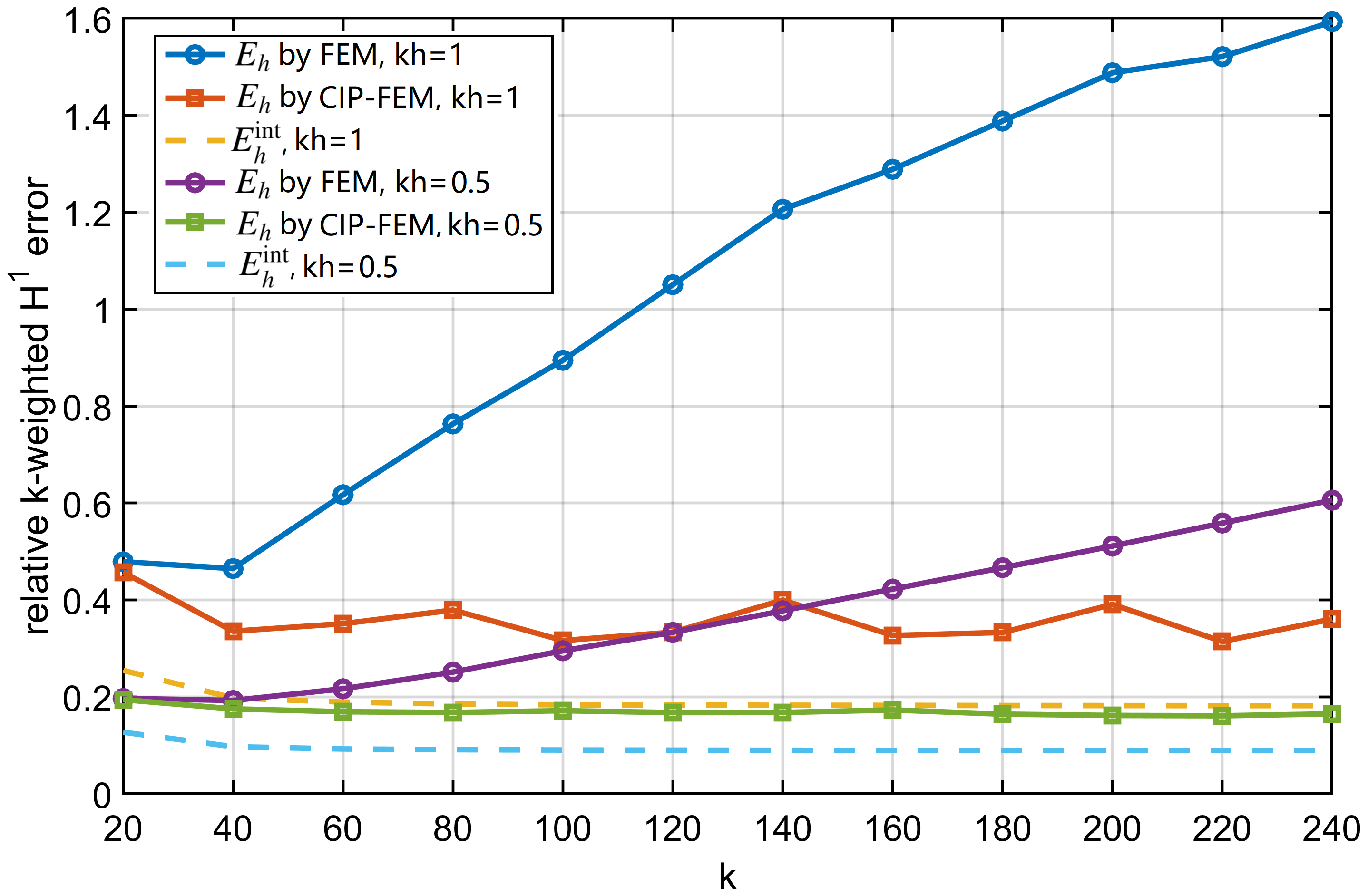}
\caption{Approximation errors and interpolation errors for \(kh=1\)
and \(kh=0.5\).}
\label{fig:obstacle-fixed-kh-error}
\end{figure}

The numerical results are qualitatively consistent with the
pre-asymptotic estimate. The interpolation error is controlled by the number of
degrees of freedom per wavelength, while the pollution error of the finite element solution grows with $k$ at fixed \(kh\). The tuned CIP penalty reduces phase errors substantially and keeps the pollution error of the same order as the interpolation error.
 
\bibliographystyle{siam}
\bibliography{references}

\end{document}